\documentclass{amsart}
\usepackage[all]{xy}
\usepackage{verbatim}
\usepackage{color}
\usepackage{amsthm}
\usepackage{amssymb}
\usepackage[colorlinks=true]{hyperref}
\usepackage{graphicx}

\numberwithin{equation}{section}

\newtheorem{theorem}[equation]{Theorem}
\newtheorem*{theorem*}{Theorem} \newtheorem{lemma}[equation]{Lemma}

\newtheorem*{conjecture*}{Mamma Conjecture}
\newtheorem*{conjecture1*}{Mamma Conjecture (revisited)}
\newtheorem{proposition}[equation]{Proposition}
\newtheorem{corollary}[equation]{Corollary}
\newtheorem*{corollary*}{Corollary}

\theoremstyle{remark}
\newtheorem{definition}[equation]{Definition}

\newtheorem{example}[equation]{Example}

\theoremstyle{remark}
\newtheorem{remark}[equation]{Remark}

\newcommand{\cA}{{\mathcal A}}
\newcommand{\cB}{{\mathcal B}}
\newcommand{\cC}{{\mathcal C}}
\newcommand{\cD}{{\mathcal D}}
\newcommand{\cE}{{\mathcal E}}
\newcommand{\cF}{{\mathcal F}}

\newcommand{\cL}{{\mathcal L}}

\newcommand{\cN}{{\mathcal N}}
\newcommand{\cO}{{\mathcal O}}

\newcommand{\cS}{{\mathcal S}}
\newcommand{\cT}{{\mathcal T}}
\newcommand{\cU}{{\mathcal U}}

\newcommand{\cG}{\mathrm{G}}
\newcommand{\cH}{\mathrm{H}}

\newcommand{\bbA}{\mathbb{A}}

\newcommand{\bbC}{\mathbb{C}}

\newcommand{\bbG}{\mathbb{G}}

\newcommand{\bbL}{\mathbb{L}}

\newcommand{\bbP}{\mathbb{P}}

\newcommand{\bbQ}{\mathbb{Q}}
\newcommand{\bbZ}{\mathbb{Z}}

\DeclareMathOperator{\SmProj}{SmProj} 

\DeclareMathOperator{\id}{id}

\DeclareMathOperator{\NChow}{NChow} 
\DeclareMathOperator{\NNum}{NNum} 

\DeclareMathOperator{\Sep}{Sep}

\DeclareMathOperator{\Num}{Num} 

\DeclareMathOperator{\Fun}{Fun} 

\newcommand{\dgcat}{\mathrm{dgcat}} 

\newcommand{\perf}{\mathrm{perf}}

\newcommand{\Chow}{\mathrm{Chow}}

\newcommand{\dg}{\mathrm{dg}}

\newcommand{\Hom}{\mathrm{Hom}}
\newcommand{\End}{\mathrm{End}}

\newcommand{\rep}{\mathrm{rep}}

\newcommand{\dgHo}{\mathrm{H}^0}

\newcommand{\Hmo}{\mathrm{Hmo}}
\newcommand{\op}{\mathrm{op}}

\newcommand{\too}{\longrightarrow}

\newcommand{\ie}{\textsl{i.e.}\ }
\newcommand{\eg}{\textsl{e.g.}}

\def\Sep{\operatorname{Sep}}

\let\oldmarginpar\marginpar
\def\marginpar#1{\oldmarginpar{\tiny #1}}

\begin{document}

\title[Equivariant noncommutative motives]{Equivariant noncommutative motives}
\author{Gon{\c c}alo~Tabuada}

\address{Gon{\c c}alo Tabuada, Department of Mathematics, MIT, Cambridge, MA 02139, USA}
\email{tabuada@math.mit.edu}
\urladdr{http://math.mit.edu/~tabuada}
\thanks{The author was partially supported by a NSF CAREER Award}

\subjclass[2000]{14A22, 14C15, 14L30, 16S35, 19L47, 55N32}
\date{\today}

\keywords{$\cG$-scheme, $\cG$-algebra, $2$-cocycle, equivariant perfect complex, semidirect product algebra, twisted group algebra, equivariant (noncommutative) motives, equivariant algebraic $K$-theory, orbifold cohomology theory, twisted projective homogeneous scheme, full exceptional collection, equivariant motivic measure, noncommutative algebraic geometry.}

\abstract{Given a finite group $\cG$, we develop a theory of $\cG$-equivariant noncommutative motives. This theory provides a well-adapted framework for the study of $\cG$-schemes, Picard groups of schemes, $\cG$-algebras, $2$-cocycles, $\cG$-equivariant algebraic $K$-theory, orbifold cohomology theory, etc. Among~other results, we relate our theory with its commutative counterpart as well as with Panin's theory. As a first application, we extend Panin's computations, concerning twisted projective homogeneous varieties, to a large class of invariants. As a second application, we prove that whenever the category of perfect complexes of a $\cG$-scheme $X$ admits a full exceptional collection of $\cG$-invariant ($\neq$ $\cG$-equivariant) objects, the $\cG$-equivariant Chow motive of $X$ is of Lefschetz type. Finally, we construct a $\cG$-equivariant motivic measure with values in the Grothendieck ring of $\cG$-equivariant noncommutative Chow motives.}
}

\maketitle
\tableofcontents



\section{Introduction}

A {\em differential graded (=dg) category $\cA$}, over a base field $k$, is a category enriched over dg $k$-vector spaces; see \S\ref{sec:dg}. Every (dg) $k$-algebra $A$ gives naturally rise to a dg category with a single object. Another source of examples is provided by schemes since the category of perfect complexes $\perf(X)$ of every quasi-compact quasi-separated $k$-scheme $X$ admits a canonical dg enhancement $\perf_\dg(X)$;~see~\S\ref{sub:perfect1}. 

Given a finite group $\cG$, we develop in \S\ref{sec:actions} a general theory of group actions on dg categories. A dg category $\cA$ equipped with a $\cG$-action is denoted by $\cG \circlearrowright \cA$ and called a {\em $\cG$-dg category}. For example, every $\cG$-scheme $X$, subgroup $\cG \subseteq \mathrm{Pic}(X)$ of the Picard group a scheme $X$, $\cG$-algebra $A$, and cohomology class $[\alpha] \in H^2(\cG, k^\times)$, gives naturally rise to a $\cG$-dg category; the cohomology classes correspond to the $\cG$-actions $\cG \circlearrowright_\alpha k$ on the base field $k$. The associated dg categories of $\cG$-equivariant objects $\cA^\cG$ are given, respectively, by equivariant perfect complexes $\perf^\cG_\dg(X)$, perfect complexes $\perf_\dg(Y)$ on a $|\cG|$-fold cover over $X$, semidirect product algebras $A \rtimes \cG$, and twisted group algebras $k_\alpha[\cG]$.

By precomposition with the functor $\cG \circlearrowright \cA \mapsto \cA^\cG$, all invariants of dg categories $E$ can be promoted to invariants of $\cG$-dg categories $E^\cG$. For example, algebraic $K$-theory leads to equivariant algebraic $K$-theory in the sense of Thomason \cite{Thomason}, and periodic cyclic homology to orbifold cohomology theory in the sense of Chen-Ruan \cite{CR}; see \S\ref{sub:additive}. In order to study all these invariants simultaneously, we develop in \S\ref{sec:equivariant} a theory of $\cG$-equivariant noncommutative motives. Among other results, we construct a symmetric monoidal functor $U^\cG\colon \dgcat^\cG_{\mathrm{sp}}(k) \to \NChow^\cG(k)$, from smooth proper $\cG$-dg categories to $\cG$-equivariant noncommutative Chow motives, which is ``initial'' among all such invariants $E^\cG$. The morphisms of $\NChow^\cG(k)$ are given in terms of the $\cG$-equivariant Grothendieck group of certain triangulated categories of bimodules. In particular, the ring of endomorphisms of the $\otimes$-unit $U^\cG(\cG \circlearrowright_1 k)$ identifies with the representation ring $R(\cG)$ of the group $\cG$.

I. Panin constructed in \cite{Panin} a certain motivic category $\cC^\cG(k)$, which mixes smooth projective $\cG$-schemes with (noncommutative) separable algebras, and performed therein several computations concerning twisted projective homogeneous varieties. In Theorem \ref{thm:bridge-Panin} we construct a fully faithful symmetric monoidal functor from $\cC^\cG(k)$ to $\NChow^\cG(k)$. As a byproduct, we extend Panin's computations to all the aforementioned invariants $E^\cG$; see Theorem \ref{thm:Panin}.

Making use of the work of Edidin-Graham \cite{EG} on equivariant intersection theory, Laterveer \cite{Laterveer}, and Iyer and M\"uller-Stack \cite{Iyer-Muller}, extended the theory of Chow motives to the $\cG$-equivariant setting. In Theorem \ref{thm:bridge}, we relate this latter theory with the one of $\cG$-equivariant noncommutative motives. Concretely, we construct a $\bbQ$-linear, fully faithful, symmetric monoidal $\Phi$ making the following diagram commute
\begin{equation}\label{eq:diagram-bridge-big}
\xymatrix{
\mathrm{SmProj}^\cG(k)^\op \ar[rrr]^-{X \mapsto \cG \circlearrowright \perf_\dg(X)} \ar[d]_-{\mathfrak{h}^\cG(-)_\bbQ} &&& \dgcat_{\mathrm{sp}}^\cG(k) \ar[d]^-{U^\cG(-)_\bbQ} \\
\Chow^\cG(k)_\bbQ \ar[d]_-\pi &&& \NChow^\cG(k)_\bbQ \ar[d]^-{(-)_{I_\bbQ}} \\
\Chow^\cG(k)_\bbQ/_{\!-\otimes \bbQ(1)} \ar[rrr]_-{\Phi} &&& \NChow^\cG(k)_{\bbQ, I_\bbQ}\,,
}
\end{equation}
where $\Chow^\cG(k)_\bbQ/_{\!-\otimes \bbQ(1)}$ stands for the orbit category (see \S\ref{sub:orbit}) and $(-)_{I_\bbQ}$ for the localization functor associated to the augmentation ideal $I \subset R(\cG) \stackrel{\mathrm{rank}}{\twoheadrightarrow} \bbZ$. Intuitively speaking, the commutative diagram \eqref{eq:diagram-bridge-big} shows that after ``$\otimes$-trivializing'' the $\cG$-equivariant Tate motive $\bbQ(1)$ and localizing at the augmentation ideal $I_\bbQ$, the commutative world embeds fully faithfully into the noncommutative world. 

The Grothendieck ring of varieties admits a $\cG$-equivariant analogue $K_0\mathrm{Var}^\cG(k)$. Although very important, the structure of this latter ring is quite mysterious. In order to capture some of its flavor, several $\cG$-equivariant motivic measures have been built. In Theorem \ref{thm:measure}, we prove that the assignment $X \mapsto U^\cG(\cG \circlearrowright \perf_\dg(X))$, with $X$ a smooth projective $\cG$-variety, gives rise to a $\cG$-equivariant motivic measure $
\mu_{\mathrm{nc}}^\cG\colon K_0\mathrm{Var}^\cG(k) \to K_0(\mathrm{NChow}^\cG(k))$
with values in the Grothendieck ring of the category of $\cG$-equivariant noncommutative Chow motives. It turns out that $\mu_{\mathrm{nc}}^\cG$ contains a lot of interesting information. For example, when $k\subseteq\bbC$, the $\cG$-equivariant motivic measure $K_0\mathrm{Var}^\cG(k) \to R_\bbC(\cG), X \mapsto \sum_i (-1)^i H^i_c(X^{\mathrm{an}},\bbC)$, factors through $\mu_{\mathrm{nc}}^\cG$; see Proposition \ref{prop:factorization-measure}.
\subsection*{Applications}
Let $X$ be a smooth projective $\cG$-scheme. In order to study it, we can proceed into two distinct directions. On one direction, we can associate to $X$ its $\cG$-equivariant Chow motive $\mathfrak{h}^\cG(X)_\bbQ$. On another direction, we can associate to $X$ its $\cG$-category of perfect complexes $\cG \circlearrowright \perf(X)$. Making use of the bridge \eqref{eq:diagram-bridge-big}, we establish the following relation\footnote{Theorem \ref{thm:via} is a far reaching generalization of the main result of \cite{Crelle}.} between these two~distinct~mathematical~objects:
\begin{theorem}\label{thm:via}
If $\perf(X)$ admits a full exceptional collection $(\cE_1, \ldots, \cE_n)$ of {\em $\cG$-invariant objects}, \ie $\sigma^\ast(\cE_i)\simeq \cE_i$ for every $\sigma \in \cG$, then there exists a choice of integers $r_1, \ldots, r_n \in  \{0, \ldots, \mathrm{dim}(X)\}$ such that 
\begin{equation}\label{eq:decomp-motivic}
\mathfrak{h}^\cG(X)_\bbQ \simeq \bbL^{\otimes r_1} \oplus \cdots \oplus \bbL^{\otimes r_n}\,,
\end{equation}
where $\bbL$ stands for the $\cG$-equivariant Lefschetz motive.
\end{theorem}
\begin{remark}
A $\cG$-equivariant object is $\cG$-invariant, but the converse does not holds!
\end{remark}
Theorem \ref{thm:via} can be applied to any $\cG$-action on projective spaces, quadrics, Grassmannians, etc; see \S\ref{sub:invariants}. Intuitively speaking, it shows that the existence of a full exceptional collection of $\cG$-invariant objects ``quasi-determines'' the $\cG$-equivariant Chow motive $\mathfrak{h}^\cG(X)_\bbQ$. The unique indeterminancy is the number of $\otimes$-powers of the $\cG$-equivariant Lefschetz motive. Note that this indeterminancy cannot by refined. For example, the categories $\perf(\mathrm{Spec}(k)\amalg \mathrm{Spec}(k))$ and $\perf(\bbP^1)$ (equipped with the trivial $\cG$-action) admit full exceptional collections of length $2$ but the corresponding $\cG$-equivariant Chow motives are distinct:
$$\mathfrak{h}^\cG(\mathrm{Spec}(k) \amalg \mathrm{Spec}(k))_\bbQ \simeq \mathfrak{h}^\cG(\mathrm{Spec}(k))_\bbQ^{\oplus 2} \not\simeq \mathfrak{h}^\cG(\mathrm{Spec}(k))_\bbQ \oplus \bbL \simeq \mathfrak{h}^\cG(\bbP^1)_\bbQ\,.$$
\begin{corollary}\label{cor:exceptional}
For every good $\cG$-cohomology theory $H^\ast_\cG$ in the sense of Laterveer \cite[Def.~1.10]{Laterveer}, we have $H^i_\cG(X)=0$ if $i$ is odd and $\sum_i \mathrm{dim}\,H^i_\cG(X)=n$.
\end{corollary}
\begin{proof}
As proved in \cite[Prop.~1.12]{Laterveer}, $H^\ast_\cG$ factors through $\Chow^\cG(k)_\bbQ$. Making use of Theorem \ref{thm:via}, we conclude that $H^\ast_\cG(X)\simeq H^\ast_\cG(\bbL)^{\otimes r_1}\oplus \cdots \oplus H^\ast_\cG(\bbL)^{\otimes r_n}$. The proof follows now from the fact that $\mathrm{dim}\,H^2_\cG(\bbL)=1$ and that $H^i_\cG(\bbL)\simeq 0$~for~$i \neq 2$.
\end{proof}
\begin{remark}
Corollary \ref{cor:exceptional} implies that the length of an hypothetical full exceptional collection of $\cG$-invariant objects is equal to $\sum_i \mathrm{dim}\, H^i_\cG(X)$. Moreover, if $H^i_\cG(X)\not\simeq 0$ for some odd integer $i$, then such a full exceptional collection cannot exist. 
\end{remark}
Theorem \ref{thm:via} shows also that the $\cG$-equivariant Chow motive $\mathfrak{h}^\cG(X)_\bbQ$ loses all the information concerning the $\cG$-action on $X$. In contrast, the $\cG$-equivariant noncomutative Chow motive $U^\cG(\cG \circlearrowright \perf_\dg(X))$ keeps track of some of the $\cG$-action! Concretely, as proved in Proposition \ref{prop:exceptional-schemes}, there exist (non-trivial) cohomology classes $[\alpha_1], \ldots, [\alpha_n] \in H^2(\cG, k^\times)$ such that 
\begin{equation}\label{eq:motivic-decomp-NC}
U^\cG(\cG \circlearrowright\perf_\dg(X))\simeq U^\cG(\cG \circlearrowright_{\alpha_1} k) \oplus \cdots \oplus U^\cG(\cG \circlearrowright_{\alpha_n} k)\,.
\end{equation}
This implies, in particular, that all the invariants $E^\cG(\cG \circlearrowright \perf_\dg(X))$ can be computed in terms of twisted group algebras $\oplus^n_{i=1} E(k_{\alpha_i}[\cG])$. Taking into account the decompositions \eqref{eq:decomp-motivic} and \eqref{eq:motivic-decomp-NC}, the $\cG$-equivariant Chow motive $\mathfrak{h}^\cG(X)_\bbQ$ and the $\cG$-equivariant noncommutative Chow motive $U^\cG(\cG \circlearrowright \perf_\dg(X))$ should be considered as complementary. While the former keeps track of the Tate twists but not of the $\cG$-action, the latter keeps track of the $\cG$-action but not of the Tate twists.
\begin{remark}
At \S\ref{sub:permutation} we discuss also the case of full exceptional collections where the objects are not $\cG$-invariant but rather permuted by the $\cG$-action.
\end{remark}

\subsection*{Notations}
Throughout the article, $k$ will denote a base field and $\cG$ a finite group. We will write $1\in \cG$ for the unit element and $|\cG|$ for the order of $\cG$. Except at \S\ref{sec:dg}-\ref{sec:actions}, we will always assume that $\mathrm{char}(k)\nmid |\cG|$. All schemes will be defined over $\mathrm{Spec}(k)$, and all adjunctions will be displayed vertically with the left (resp. right) adjoint on the left (resp. right) hand side.
\section{Background on dg categories}\label{sec:dg}
Let $(\cC(k),\otimes, k)$ be the symmetric monoidal category of dg $k$-vector spaces; we use cohomological notation. A {\em dg category $\cA$} is a category enriched over $\cC(k)$ and a {\em dg functor} $F:\cA\to \cB$ is  a functor enriched over $\cC(k)$; consult Keller's ICM survey \cite{ICM-Keller}. Let us write $\dgcat(k)$ for the category of small dg categories and dg functors. 

Let $\cA$ be a dg category. The opposite dg category $\cA^\op$ has the same objects as $\cA$ and dg $k$-vector spaces $\cA^\op(x,y):=\cA(y,x)$. The category $\mathrm{Z}^0(\cA)$ has the same objects as $\cA$ and morphisms $\mathrm{Z}^0(\cA)(x,y):=Z^0(\cA(x,y))$, where $Z^0(-)$ denotes the $0^{\mathrm{th}}$-cycles functor. The category $\dgHo(\cA)$ has the same objects as $\cA$ and morphisms $\dgHo(\cA)(x,y):=H^0(\cA(x,y))$, where $H^0(-)$ denotes the $0^{\mathrm{th}}$-cohomology functor.
\subsection{Dg equivalences}
Let $\cA$ and $\cB$ be two dg categories. Recall from \cite[\S2.3]{ICM-Keller} the definition of the dg category of dg functors $\Fun_\dg(\cA,\cB)$. Given dg functors $F,G\colon \cA \to \cB$, a natural transformation of dg functors $\epsilon\colon F \Rightarrow G$ corresponds to an element of $\mathrm{Z}^0(\Fun_\dg(\cA,\cB))(F,G)$. When $\epsilon$ is invertible, we call it a {\em natural isomorphism of dg functors}. A dg functor $F\colon \cA \to \cB$ is called a {\em dg equivalence} if there exists a dg functor $G\colon \cB \to \cA$ and natural isomorphisms of dg functors $F\circ G \Rightarrow \id$ and $\id \Rightarrow G \circ F$. Equivalently, the dg functor $F$ is fully faithful and the induced functor $\mathrm{Z}^0(F)$ is essentially surjective.
\subsection{Dg modules}
Let $\cA$ be a small dg category. A (right) {\em dg $\cA$-module} is a dg functor $M\colon \cA^\op \to \cC_\dg(k)$ with values in the dg category of dg $k$-vector spaces. Let us write $\cC(\cA)$ for the category of dg $\cA$-modules and $\cC_\dg(\cA):=\Fun_\dg(\cA^\op,\cC_\dg(k))$ its dg enhancement. The latter dg category comes equipped with the Yoneda dg functor $\cA \to \cC_\dg(\cA), x \mapsto \cA(-,x)$. Following \cite[\S3.2]{ICM-Keller}, the {\em derived category $\cD(\cA)$ of $\cA$} is defined as the localization of $\cC(\cA)$ with respect to the (objectwise) quasi-isomorphisms. This category is triangulated and admits arbitrary direct sums. Let us write $\cD_c(\cA)$ for the full subcategory of compact objects. In the same vein, let $\cC_{c, \dg}(\cA)$ be the full dg subcategory of $\cC_\dg(\cA)$ consisting of those dg $\cA$-modules which belong to $\cD_c(\cA)$. By construction, we have $\dgHo(\cC_{c, \dg}(\cA))\simeq \cD_c(\cA)$. 
\subsection{Morita equivalences}\label{sub:Morita}
A dg functor $F:\cA\to \cB$ is called a {\em Morita equivalence} if the restriction functor $\cD(\cB) \to \cD(\cA)$ is an equivalence of (triangulated) categories. An example is the Yoneda dg functor $\cA \to \cC_{c, \dg}(\cA)$. As proved in \cite[Thm.~5.3]{Additive}, the category $\dgcat(k)$ admits a Quillen model structure whose weak equivalences are the Morita equivalences. Let $\Hmo(k)$ be the homotopy category.
\subsection{Product and coproduct}
The product $\cA\times\cB$, resp. coproduct $\cA\amalg\cB$, of two small dg categories $\cA$ and $\cB$ is defined as follows: the set of objects is the cartesian product, resp. disjoint union, of the sets of objects and the dg $k$-vector spaces $(\cA\times \cB)((x,w),(y,z))$, resp. $(\cA\amalg \cB)(x,y)$, are given by $\cA(x,y) \times \cB(w,z)$, resp. by $\cA(x,y)$ if $x, y \in \cA$, by $\cB(x,y)$ if $x, y \in \cB$, and by $0$ otherwise.
\subsection{Tensor product}
The tensor product $\cA\otimes\cB$ of two small dg categories $\cA$ and $\cB$ is defined as follows: the set of objects is the cartesian product of the sets of objects and the dg $k$-vectors spaces $(\cA\otimes\cB)((x,w),(y,z))$ are given by $\cA(x,y) \otimes \cB(w,z)$. As explained in \cite[\S2.3]{ICM-Keller}, this construction gives rise to a symmetric monoidal structure on $\dgcat(k)$, which descends to the homotopy category $\Hmo(k)$. 

\subsection{Dg bimodules}
A {\em dg $\cA\text{-}\cB$-bimodule} is a dg functor $\mathrm{B}\colon \cA\otimes \cB^\op \to \cC_\dg(k)$ or equivalently a dg $(\cA^\op \otimes \cB)$-module. An example is the dg $\cA\text{-}\cB$-bimodule
\begin{eqnarray}\label{eq:bimodule2}
{}_F\mathrm{B}:\cA\otimes \cB^\op \too \cC_\dg(k) && (x,z) \mapsto \cB(z,F(x))
\end{eqnarray}
associated to a dg functor $F\colon\cA\to \cB$. Let us write $\rep(\cA,\cB)$ for the full triangulated subcategory $\cD(\cA^\op \otimes \cB)$ consisting of those dg $\cA\text{-}\cB$-bimodules $\mathrm{B}$ such that for every $x \in \cA$ the dg $\cB$-module $\mathrm{B}(x,-)$ belongs to $\cD_c(\cB)$. In the same vein, let $\rep_\dg(\cA,\cB)$ be the full dg subcategory of $\cC_\dg(\cA^\op \otimes \cB)$ consisting of those dg $\cA\text{-}\cB$-bimodules which belong to $\rep(\cA,\cB)$. By construction, $\dgHo(\rep_\dg(\cA,\cB))\simeq \rep(\cA,\cB)$.

\subsection{Smooth proper dg categories}\label{sub:smooth}
Following Kontsevich \cite{IAS,ENS,Miami,finMot}, a dg category $\cA$ is called {\em smooth} if the dg $\cA\text{-}\cA$-bimodule ${}_{\id}\mathrm{B}$ belongs to the triangulated category $\cD_c(\cA^\op\otimes \cA)$ and {\em proper} if $\sum_i \mathrm{dim}\, H^i\cA(x,y)< \infty$ for any ordered pair of objects $(x,y)$. Examples include the finite dimensional $k$-algebras of finite global dimension (when $k$ is perfect) as well as the dg categories $\perf_\dg(X)$ associated to smooth proper schemes $X$. Given smooth proper dg categories $\cA$ and $\cB$, the associated dg categories $\cA\times \cB$, $\cA\amalg \cB$, and $\cA\otimes \cB$, are also smooth proper. Let us write $\dgcat_{\mathrm{sp}}(k)$ for the full subcategory of $\dgcat(k)$ consisting of the smooth proper dg categories.

\section{Equivariant perfect complexes}\label{sub:perfect}
Let $\cE$ be an abelian (or exact) category. Following Keller \cite[\S4.4]{ICM-Keller}, the {\em derived dg category $\cD_\dg(\cE)$ of $\cE$} is defined as the dg quotient $\cC_\dg(\cE)/\mathcal{A}c_\dg(\cE)$ of the dg category of complexes over $\cE$ by its full dg subcategory of acyclic complexes.
\subsection{Perfect complexes}\label{sub:perfect1}
Let $X$ be a quasi-compact quasi-separated scheme. We write $\mathrm{Mod}(X)$ for the Grothendieck category of $\cO_X$-modules, $\cD(X)$ for the derived category $\cD(\mathrm{Mod}(X))$, and $\cD_\dg(X)$ for the dg category $\cD_\dg(\cE)$ with $\cE:=\mathrm{Mod}(X)$. Recall that a complex of $\cO_X$-modules $\cF \in \cD(X)$ is called {\em perfect} if there exists a covering $X=\bigcup_i V_i$ of $X$ by affine open subschemes $V_i \hookrightarrow X$ such that for every $i $ the restriction $\cF_{|V_i}$ of $\cF$ to $V_i$ is quasi-isomorphic to a bounded complex of finitely generated projective $\cO_{|V_i}$-modules. Let us write $\perf(X)$, resp. $\perf_\dg(X)$, for the full triangulated subcategory, resp. full dg subcategory, of perfect complexes.
\subsection{Equivariant perfect complexes}
Let $X$ be a quasi-compact quasi-separated $\cG$-scheme. A {\em $\cG$-equivariant $\cO_X$-module $\cF$} is a $\cO_X$-module equipped with a family of isomorphisms $\theta_\sigma\colon \cF \to \sigma^\ast(\cF), \sigma \in \cG$, with $\theta_1=\id$, such that the compositions 
$$ \cF \stackrel{\theta_\rho}{\too} \rho^\ast(\cF) \stackrel{\rho^\ast(\theta_\sigma)}{\too} \rho^\ast(\sigma^\ast(\cF))$$
are equal to $\theta_{\rho\sigma}\colon \cF \to (\rho\sigma)^\ast(\cF)$ for every $\sigma, \rho \in \cG$. A {\em morphism of $\cG$-equivariant $\cO_X$-modules $(\cF, \theta_\sigma) \to (\mathcal{G},\theta_\sigma)$} is a morphism of $\cO_X$-modules $f\colon \cF \to \mathcal{G}$ such that $\theta_\sigma \circ f = \sigma^\ast(f) \circ \theta_\sigma$
for every $\sigma \in \cG$. We write $\mathrm{Mod}^\cG(X)$ for the Grothendieck category of $\cG$-equivariant $\cO_X$-modules, $\cD^\cG(X)$ for the derived category $\cD(\mathrm{Mod}^\cG(X))$, and $\cD^\cG_\dg(X)$ for the dg category $\cD_\dg(\cE)$ with $\cE:=\mathrm{Mod}^\cG(X)$. A complex of $\cG$-equivariant $\cO_X$-modules $\cF \in \cD^\cG(X)$ is called a {\em $\cG$-equivariant perfect complex} if the underlying complex of $\cO_X$-modules is perfect. Let us write $\perf^\cG(X)$, resp. $\perf^\cG_\dg(X)$, for the full triangulated subcategory, resp. full dg subcategory, of $\cG$-equivariant perfect complexes.
\subsection{Twisted equivariant perfect complexes}
\begin{definition}[$2$-cocycle]
A map $\alpha \colon \cG \times \cG \to k^\times$ is called a {\em $2$-cocycle} if $\alpha(1,\sigma)=\alpha(\sigma, 1)=1$ and $\alpha(\rho, \alpha)\alpha(\tau,\rho\sigma)=\alpha(\tau, \rho)\alpha(\tau \rho, \sigma)$ for every $\sigma, \rho, \tau \in \cG$.
\end{definition}
Let $X$ be a quasi-compact quasi-separated $\cG$-scheme and $\alpha$ a {\em $2$-cocycle}. An {\em $\alpha$-twisted $\cG$-equivariant $\cO_X$-module $\cF$} is a $\cO_X$-module equipped with a family of isomorphisms $\theta_\sigma\colon \cF \to \sigma^\ast(\cF), \sigma \in \cG$, with $\theta_1=\id$, such that the compositions 
$$\cF \stackrel{\theta_\rho}{\too} \rho^\ast(\cF) \stackrel{\rho^\ast(\theta_\sigma)}{\too} \rho^\ast(\sigma^\ast(\cF))$$
are equal to $\alpha(\rho, \sigma) \theta_{\rho\sigma}\colon \cF \to (\rho\sigma)^\ast(\cF)$ for every $\sigma, \rho \in \cG$. A {\em morphism of $\alpha$-twisted $\cG$-equivariant $\cO_X$-modules} $(\cF, \theta_\sigma) \to (\mathcal{G}, \theta_\sigma)$ is a morphism of $\cO_X$-modules $f\colon \cF \to \mathcal{G}$ such that $\theta_\sigma \circ f = \sigma^\ast(f) \circ \theta_\sigma$
for every $\sigma \in \cG$. We write $\mathrm{Mod}^{\cG,\alpha}(X)$ for the Grothendieck category of $\alpha$-twisted $\cG$-equivariant $\cO_X$-modules, $\cD^{\cG,\alpha}(X)$ for the derived category $\cD(\mathrm{Mod}^{\cG,\alpha}(X))$, and $\cD_\dg^{\cG, \alpha}(X)$ for the dg category $\cD_\dg(\cE)$ with $\cE:=\mathrm{Mod}^{\cG,\alpha}(X)$. A complex of $\alpha$-twisted $\cG$-equivariant $\cO_X$-modules $\cF \in \cD^{\cG,\alpha}(X)$ is called a {\em $\alpha$-twisted $\cG$-equivariant perfect complex} if the underlying complex of $\cO_X$-modules is perfect. Let us write $\perf^{\cG,\alpha}(X)$, resp. $\perf_\dg^{\cG,\alpha}(X)$, for the full triangulated subcategory, resp. full dg subcategory, of $\alpha$-twisted $\cG$-equivariant perfect complexes. 
\section{Group actions on dg categories}\label{sec:actions}
In this section we develop a general theory of group actions on dg categories. Following Deligne \cite{Deligne} and Elagin \cite{Elagin}, we start by introducing the following notion\footnote{P. Seidel introduced in \cite{Seidel, Seidel1} the notion of a circle action on a $A_\infty$-category.}: 
\begin{definition}
A (left) {\em $\cG$-action on a dg category $\cA$} consists of the data:
\begin{itemize}
\item[(i)] a family of dg equivalences $\phi_\sigma\colon \cA \to \cA, \sigma \in \cG$, with $\phi_1=\id$;
\item[(ii)] a family of natural isomorphisms of dg functors $\epsilon_{\rho,\sigma}\colon \phi_\rho \circ \phi_\sigma \Rightarrow \phi_{\rho\sigma}, \sigma, \rho \in \cG$, with $\epsilon_{1,\sigma}=\epsilon_{\sigma,1}=\id$, such that the equality $\epsilon_{\tau \rho, \sigma} \circ (\epsilon_{\tau, \rho}\circ \phi_\sigma)= \epsilon_{\tau, \rho\sigma} \circ (\phi_\tau \circ \epsilon_{\rho, \sigma})$
holds for every $\sigma, \rho, \tau \in \cG$.
\end{itemize}
\end{definition}
Throughout the article, a dg category $\cA$ equipped with a $\cG$-action will be denoted by $\cG \circlearrowright \cA$ and will be called a {\em $\cG$-dg category}.
\begin{example}[$\cG$-schemes]\label{ex:G-schemes}
Given a quasi-compact quasi-separated $\cG$-scheme $X$, the dg category $\perf_\dg(X)$ inherits a $\cG$-action induced by the pull-back dg equivalences $\phi_\sigma:=\sigma^\ast$; consult Elagin \cite{Elagin} and Sosna \cite{Sosna} for details.
\end{example}
\begin{example}[Line bundles]\label{ex:lines}
Let $X$ be a quasi-compact quasi-separated scheme. In the case where $\cG$ can be realized as a subgroup of the Picard group $\mathrm{Pic}(X)$, the dg category $\perf_\dg(X)$ inherits a $\cG$-action induced by the dg equivalences $\phi_\sigma:=-\otimes_{\cO_X}\cL_\sigma$, where $\cL_\sigma$ stands for the invertible line bundle associated to $\sigma \in \cG$; consult Elagin \cite{Elagin} and Sosna \cite{Sosna} for details.\end{example}
\begin{example}[$\cG$-algebras]\label{ex:G-algebras}
Given a $\cG$-action on a (dg) algebra $A$, the associated dg category with a single object inherits a $\cG$-action with $\epsilon_{\rho, \sigma}:=\id$.
\end{example}
\begin{example}[$2$-cocycles]\label{ex:2-cocycles}
Given a $2$-cocycle $\alpha\colon \cG \times \cG \to k^\times$, the dg category $k$ inherits a $\cG$-action given by $\phi_\sigma:=\id$ and $\epsilon_{\rho, \sigma}:=\alpha(\rho,\sigma)$. In what follows, we will denote this $\cG$-dg category by $\cG \circlearrowright_\alpha k$. Note that these are all the possible $\cG$-actions on the dg category $k$.
\end{example}
\begin{remark}[Opposite dg category]\label{rk:opposite}
Let $\cG \circlearrowright \cA$ be a $\cG$-dg category. The opposite dg category $\cA^\op$ inherits a $\cG$-action given by the dg equivalences $\phi_\sigma$ and by the natural isomorphisms of dg functors $\epsilon^{-1}_{\rho, \sigma}$.
\end{remark}
\begin{remark}[Tensor product]\label{rk:tensorproduct}
Let $\cG \circlearrowright \cA$ and $\cG \circlearrowright\cB$ be two $\cG$-dg categories. The tensor product $\cA \otimes \cB$ inherits a $\cG$-action given by the dg equivalences $\phi_\sigma \otimes \phi_\sigma$ and by the natural isomorphisms of dg equivalences $\epsilon_{\rho, \sigma} \otimes \epsilon_{\rho, \sigma}$. Similarly, the product $\cA\times \cB$ inherits a $\cG$-action given by the dg equivalences $\phi_\sigma \times \phi_\sigma$ and by the natural isomorphisms of dg functors $\epsilon_{\rho, \sigma} \times \epsilon_{\rho, \sigma}$.
\end{remark}
\begin{remark}[Dg category of dg functors]\label{rk:dgcatdgfun}
Let $\cG \circlearrowright \cA$ and $\cG \circlearrowright\cB$ be two $\cG$-dg categories. The dg category of dg functors $\Fun_\dg(\cA,\cB)$ inherits a $\cG$-action given by the dg equivalences $F \mapsto \phi_\sigma \circ F\circ \phi_{\sigma^{-1}}$ and by the natural isomorphisms of dg functors induced from $\epsilon_{\sigma^{-1}, \rho^{-1}}$ and $\epsilon_{\rho,\sigma}$. 
\end{remark}
\begin{remark}[Dg modules]\label{rk:dgmodules}
Let $\cG \circlearrowright \cA$ be a small $\cG$-dg category, and $\cC_\dg(k)$ the dg category dg $k$-vector spaces equipped with the trivial $\cG$-action. Thanks to Remarks \ref{rk:opposite} and \ref{rk:dgcatdgfun}, the dg category of dg $\cA$-modules $\cC_\dg(\cA):=\Fun_\dg(\cA^\op, \cC_\dg(k))$ inherits a $\cG$-action, which restricts to $\cC_{c, \dg}(\cA)$.
\end{remark}
\begin{remark}[Dg bimodules]\label{rk:dgbimodules}
Let $\cG \circlearrowright \cA$ and $\cG \circlearrowright \cB$ be two small $\cG$-dg categories, and $\cC_\dg(k)$ the dg category of dg $k$-vector spaces equipped with the trivial $\cG$-action. Thanks to Remarks \ref{rk:opposite}-\ref{rk:dgcatdgfun}, the dg category of dg $\cA\text{-}\cB$-bimodules $\cC_\dg(\cA^\op \otimes \cB):=\Fun_\dg(\cA\otimes \cB^\op, \cC_\dg(k))$ inherits a $\cG$-action, which restricts to $\rep_\dg(\cA,\cB)$.
\end{remark}
\begin{definition}
A {\em $\cG$-equivariant dg functor $\cG \circlearrowright \cA \to \cG \circlearrowright \cB$} consists of the data:
\begin{itemize}
\item[(i)] a dg functor $F\colon \cA \to \cB$;
\item[(ii)] a family of natural isomorphisms of dg functors $\eta_\sigma\colon F \circ \phi_\sigma \Rightarrow \phi_\sigma \circ F, \sigma \in \cG$, such that $\eta_{\rho\sigma} \circ (F \circ \epsilon_{\rho,\sigma})= (\epsilon_{\rho, \sigma}\circ F) \circ (\phi_\rho \circ \eta_\sigma)\circ (\eta_\rho \circ \phi_\sigma)$ for every $\sigma, \rho \in \cG$.
\end{itemize}
\end{definition}
A $\cG$-equivariant dg functor with $F$ a Morita equivalence is called a {\em $\cG$-equivariant Morita equivalence}. For example, given a small $\cG$-dg category $\cG \circlearrowright \cA$, the Yoneda dg functor $\cA \to \cC_{c, \dg}(\cA), x \mapsto \cA(-,x)$, is a $\cG$-equivariant Morita equivalence.

Let us denote by $\dgcat^\cG(k)$ the category whose objects are the small $\cG$-dg categories and whose morphisms are the $\cG$-equivariant dg functors. Given $\cG$-equivariant dg functors $F\colon \cG \circlearrowright \cA \to \cG \circlearrowright \cB$ and $G\colon \cG \circlearrowright \cB \to \cG \circlearrowright \cC$, their composition is defined as $(G\circ F, (\eta_\sigma \circ F)\circ (G \circ \eta_\sigma))$. The category $\dgcat^\cG(k)$ carries a symmetric monoidal structure given by $(\cG \circlearrowright \cA) \otimes (\cG \circlearrowright \cB):= \cG \circlearrowright (\cA\otimes \cB)$. This monoidal structure is closed, with internal-Homs given by $\cG \circlearrowright \Fun_\dg(\cA,\cB)$. 

By construction, we have the restriction functor
\begin{eqnarray}\label{eq:forgetful}
\dgcat^\cG(k) \too \dgcat(k) && \cG \circlearrowright \cA \mapsto \cA
\end{eqnarray}
as well as the trivial $\cG$-action functor
\begin{eqnarray}\label{eq:trivial}
\dgcat(k) \too \dgcat^\cG(k) && \cA \mapsto \cG \circlearrowright_1 \cA\,,
\end{eqnarray}
where $\cG \circlearrowright_1 \cA$ is equipped with the $\cG$-action given by $\phi_\sigma:=\id$ and $\epsilon_{\rho,\sigma}:=\id$. Note that \eqref{eq:forgetful}-\eqref{eq:trivial} are symmetric monoidal and that \eqref{eq:trivial} is faithful but~not~full. 

\begin{proposition}\label{prop:2-cocycles}
Let $\alpha$ and $\beta$ be two $2$-cocycles. The $\cG$-dg categories $\cG \circlearrowright_\alpha k$ and $\cG \circlearrowright_\beta k$ are isomorphic in $\dgcat^\cG(k)$ if and only if the cohomology classes $[\alpha]$ and $[\beta]$ are the same in $H^2(\cG,k^\times)$.
\end{proposition}
\begin{proof}
Recall that a map $\delta\colon \cG \to k^\times$ is called a {\em coboundary between $\alpha$ and $\beta$} if $\delta(\rho \sigma) \alpha(\rho, \sigma)=\delta(\sigma) \delta(\rho) \beta(\rho,\sigma)$ for every $\sigma, \rho \in \cG$. If such a coboundary exists, then we can consider the $\cG$-equivariant dg functor $\cG \circlearrowright_\alpha k \to \cG \circlearrowright_\beta k$ defined by $F:=\id$ and $\eta_\sigma:=\delta(\sigma)$. This $\cG$-equivariant dg functor is an isomorphism in $\dgcat^\cG(k)$, with inverse given by $G := \id$ and $\eta_\sigma:= \delta(\sigma)^{-1}$. Therefore, we conclude that if $[\alpha]=[\beta]$ in $H^2(\cG,k^\times)$, then the $\cG$-dg categories $\cG \circlearrowright_\alpha k$ and $\cG \circlearrowright_\beta k$ are isomorphic in $\dgcat^\cG(k)$. Conversely, suppose that $\cG \circlearrowright_\alpha k$ and $\cG \circlearrowright_\beta k$ are isomorphic in $\dgcat^\cG(k)$. An isomorphism is necessarily given by the identity dg functor $F:=\id$ and by a map $\delta\colon \cG \to k^\times$ (corresponding to the natural isomorphisms of dg functors $\eta_\sigma$) such that $\delta(\rho\sigma)\alpha(\rho, \sigma)=\delta(\sigma) \delta(\rho) \beta(\rho,\sigma)$ for every $\sigma, \rho \in \cG$, \ie by a coboundary between $\alpha$ and $\beta$. This concludes the proof.
\end{proof}
\begin{example}
When $k=\bbC$, we have the computations
$$ H^2(C_n,\bbC^\times)\simeq 0 \quad H^2(S_n,\bbC^\times) \simeq \begin{cases} 0 & n \leq 3 \\ C_2 & n \geq 4 \end{cases} \quad H^2(A_n,\bbC^\times)\simeq  \begin{cases} 0 & n \leq 3 \\ C_2 & n \geq 4\neq 6,7\\ C_6 & n=6,7 \end{cases}$$
$$ H^2(D_{2n}, \bbC^\times) \simeq  \begin{cases} 0 & n \,\, \mathrm{odd} \\ C_2 & n \,\,\mathrm{even} \end{cases} \quad \quad H^2(E_{p^n}, \bbC^\times)\simeq E_{p^{\frac{n(n-1)}{2}}}\,,$$
where $C_n$ stands for the cyclic group of order $n$, $S_n$ for the symmetric group on $n$ letters, $A_n$ for the alternating group on $n$ letters, $D_{2n}$ for the dihedral group associated to a polygon with $n$ sides, and $E_{p^n}$ for the elementary abelian group of order $p^n$. In general, $H^2(\cG,k^\times)$ is finite and a $\bbZ/|\cG|$-module.
\end{example}

%
Let us denote by $\mathrm{Pic}(\dgcat^\cG(k))$ the Picard group of the category $\dgcat^\cG(k)$.
\begin{proposition}\label{prop:injective}
We have an injective group homomorphism
\begin{eqnarray*}
H^2(\cG,k^\times) \too \mathrm{Pic}(\dgcat^\cG(k)) && [\alpha] \mapsto \cG \circlearrowright_\alpha k\,.
\end{eqnarray*}
\end{proposition}
\begin{proof}
Given any two $2$-cocycles $\alpha$ and $\beta$, the $\cG$-dg category $(\cG \circlearrowright_\alpha k) \otimes (\cG \circlearrowright_\beta k)$ is isomorphic to $\cG \circlearrowright_{\alpha \beta} k$. This implies that $\cG \circlearrowright_\alpha k$ is an element of the Picard group, with $\otimes$-inverse $\cG \circlearrowright_{\alpha^{-1}}k$. The proof follows now from Proposition~\ref{prop:2-cocycles}.
\end{proof}
%
%
%
\subsection{Equivariant objects}
Let $\cG \circlearrowright \cA$ be a $\cG$-dg category.
\begin{definition}\label{def:G-equivariant}
\begin{itemize}
\item[(i)] A {\em $\cG$-equivariant object in $\cG \circlearrowright \cA$} consists of an object $x \in \cA$ and of a family of closed degree zero isomorphisms $\theta_\sigma\colon x \to \phi_\sigma(x), \sigma \in \cG$, with $\theta_1=\id$, such that the compositions 
$$
x \stackrel{\theta_\rho}{\too} \phi_\rho(x) \stackrel{\phi_\rho(\theta_\sigma)}{\too} \phi_\rho(\phi_\sigma (x)) \stackrel{\epsilon_{\rho,\sigma}(x)}{\too} \phi_{\rho\sigma}(x)
$$
are equal to $\theta_{\rho\sigma}\colon x \to \phi_{\rho\sigma}(x)$ for every $\sigma, \rho \in \cG$.
\item[(ii)] A {\em morphism of $\cG$-equivariant objects $(x,\theta_\sigma) \to (y,\theta_\sigma)$} is an element $f$ of the dg $k$-vector space $\cA(x,y)$ such that $\theta_\sigma \circ f = \phi_\sigma(f) \circ \theta_\sigma$ for every $\sigma \in \cG$.
\end{itemize}
Let us write $\cA^\cG$ for the dg category of $\cG$-equivariant objects in $\cG \circlearrowright \cA$.
\end{definition}
From a topological viewpoint, the dg category $\cA^\cG$ may be understood as the ``homotopic fixed points'' of the $\cG$-action on $\cA$.
\begin{example}[Equivariant perfect complexes]\label{ex:G-equivariant}
Let $\cG \circlearrowright \perf_\dg(X)$ be as in Example~\ref{ex:G-schemes}. When $\mathrm{char}(k)\nmid |G|$, Elagin proved in \cite[Thm.~1.1]{Elagin}\cite[Thm.~9.6]{Elagin-cohomological} that $\perf_\dg(X)^\cG$ is Morita equivalent to the dg category of $\cG$-equivariant perfect complexes $\perf_\dg^\cG(X)$. In some cases, the latter dg category admits a geometric description in terms of a resolution of the singular quotient $X/\cG$:
\begin{itemize}
\item[(i)] Let $X$ be a smooth $\cG$-scheme of dimension $\leq 3$ such that $\cG_x \subset \mathrm{SL}(T_X(x))$ for all closed points $x \in X$. In these cases, Bridgeland, King, and Reid, constructed in  \cite{BRK} a crepant resolution $Y \to X/\cG$ (using a component of the Hilbert scheme of $\cG$-clusters) and proved that $\perf^\cG_\dg(X)$ is Morita equivalent to $\perf_\dg(Y)$. For example, when the cyclic group $\cG=C_2$ acts by the involution $a\mapsto -a$ on an abelian surface $S$, the crepant=minimal resolution of the quotient $\cS/C_2$ is given by the Kummer surface $\mathrm{Km}(S)$.
\item[(ii)] Let $V$ be a symplectic vector space and $\cG \subset \mathrm{Sp}(V)$ a finite subgroup. Assuming the existence of a crepant resolution $Y \to V/\cG$, Bezrukavnikov and Kaledin proved in \cite{BKaledin} that $\perf_\dg^\cG(V)$ is Morita equivalent to $\perf_\dg(Y)$. 
\end{itemize}
A well known conjecture of Reid asserts that whenever the quotient $X/\cG$ admits a crepant resolution $Y$, the dg categories $\perf^\cG_\dg(X)$ and $\perf_\dg(Y)$ are Morita equivalent. Besides the preceding cases (i)-(ii), this conjecture remains wide open.
\end{example}
\begin{example}[Covering spaces]\label{ex:covering}
Let $\cG \circlearrowright \perf_\dg(X)$ be as in Example \ref{ex:lines}. Consider the relative spectrum $Y:=\mathrm{Spec}_X(\oplus_{\sigma \in \cG}\cL_\sigma^{-1})$, which is a non-ramified $|G|$-fold cover of $X$. When $\mathrm{char}(k)\nmid |G|$, Elagin proved in \cite[Thm.~1.2]{Elagin} that $\perf_\dg(X)^\cG$ is Morita equivalent to $\perf_\dg(Y)$. In the particular case where $X$ is an Enriques surface, $\cG=C_2$ is the cyclic group of order $2$, and $\cL$ is the canonical bundle of $X$, the $2$-fold cover $Y$ of $X$ is known to be a $K3$-surface.
\end{example}
\begin{example}[Semidirect product algebras]\label{ex:semidirect}
Let $\cG \circlearrowright A$ be as in Example \ref{ex:G-algebras}. As mentioned in Remark \ref{rk:dgmodules}, the dg category $\cC_{c,\dg}(A)$ inherits a $\cG$-action. Moreover, it admits direct sums and $\dgHo(\cC_{c, \dg}(A))\simeq \cD_c(A)$ is an idempotent complete triangulated category. Furthermore, the dg $A$-module $A$ generates the triangulated category $\cD_c(A)$. Making use of Lemma \ref{lem:generator} below, we conclude that when $\mathrm{char}(k)\nmid |G|$, the dg category $\cC_{c, \dg}(A)^\cG$ is Morita equivalent to the dg algebra of endomorphisms of the $\cG$-equivariant object $(\oplus_{\rho \in \cG}\phi_\rho(A), \theta_\sigma)$. A simple computation, using the fact that $\theta_\sigma=\id$, shows that this (dg) $k$-algebra is isomorphic to the semidirect product (dg) algebra $A \rtimes \cG$.
\end{example}
\begin{example}[Twisted group algebras]\label{ex:twisted}
Let $\cG \circlearrowright_\alpha k$ be as in Example \ref{ex:2-cocycles}. Similarly to Example \ref{ex:semidirect}, the dg category $\cC_{c, \dg}(k)^\cG$ is Morita equivalent to the (dg) algebra of endomorphisms of the $\cG$-equivariant object $(\oplus_{\rho \in \cG} \phi_\rho(k), \theta_\sigma)$. A simple computation, using the fact that $\phi_\rho(k)=k$, shows that this (dg) $k$-algebra is isomorphic to the twisted group algebra $k_\alpha[\cG]$. Roughly speaking, the twisted group algebras are the ``homotopic fixed points'' of the $\cG$-actions on the dg category $k$.
\end{example}
\begin{lemma}\label{lem:generator}
Assume that $\mathrm{char}(k)\nmid |G|$. Let $\cG \circlearrowright \cA$ be a $\cG$-dg category such that $\cA$ admits direct sums and $\dgHo(\cA)$ is an idempotent complete triangulated category. If $x \in \cA$ generates the triangulated category $\dgHo(\cA)$, then the dg category $\cA^\cG$ is Morita equivalent to the dg algebra of endomorphisms of the $\cG$-equivariant object $(\oplus_{\rho \in \cG}\phi_\rho(x), \theta_\sigma) \in \cA^\cG$, where $\theta_\sigma$ is given by the collection of isomorphisms~$\epsilon_{\sigma, \rho}(x)^{-1}$.
%
%
\end{lemma}
\begin{proof}
As proved in \cite[Lem.~8.6]{Elagin}, the category $\cD_c(\cA^\cG)\simeq \dgHo(\cC_{c, \dg}(\cA^\cG))$ is equivalent to $\dgHo(\cC_{c, \dg}(\cA))^\cG \simeq \cD_c(\cA)^\cG$. Moreover, following \cite[Lem.~3.8]{Elagin}, we have the adjunction of categories:
\begin{equation}\label{eq:adjunction}
\xymatrix{
\cD_c(\cA)^\cG\ar@<1ex>[d]^-{(M,\theta_\sigma) \mapsto M}\\
\cD_c(\cA) \ar@<1ex>[u]^-{M\mapsto (\oplus_\rho \phi_\rho(M), \theta_\sigma)}\,.
}
\end{equation}
Using the fact that the right adjoint functor is conservative, we conclude from \eqref{eq:adjunction} that if $x$ generates the triangulated category $\dgHo(\cA)\simeq \cD_c(\cA)$, then the image of the $\cG$-equivariant object $(\oplus_{\rho \in \cG}\phi_\rho(x), \theta_\sigma) \in \cA^\cG$ under the Yoneda dg functor generates the triangulated category $\cD_c(\cA^\cG)$. This implies that the dg category $\cA^\cG$ is Morita equivalence to the dg algebra of endomorphisms of $(\oplus_{\rho \in \cG}\phi_\rho(x), \theta_\sigma)$.
\end{proof}
\begin{remark}[$\cG$-equivariant dg functors]\label{rk:bijection}
Let $\cG \circlearrowright \cA$ and $\cG \circlearrowright \cB$ be two dg categories. The assignment $(F,\eta_\sigma) \mapsto (F,(\eta_\sigma \circ \phi_{\sigma^{-1}})\circ (F \circ \epsilon^{-1}_{\sigma, \sigma^{-1}}))$ establishes a bijection between the set of $\cG$-equivariant dg functors $\cG \circlearrowright \cA \to \cG \circlearrowright \cB$ and the set of of $\cG$-equivariant objects in $\cG \circlearrowright \Fun_\dg(\cA,\cB)$ (see Remark \ref{rk:dgcatdgfun}). Its inverse is given by the assignment $(F,\theta_\sigma) \mapsto (F, (\phi_\sigma \circ F \circ \epsilon_{\sigma^{-1},\sigma}) \circ (\theta_\sigma \circ \phi_\sigma))$.
\end{remark}


Given a $\cG$-equivariant dg functor $F\colon \cG \circlearrowright \cA \to \cG \circlearrowright \cB$, the  assignment $(x,\theta_\sigma) \mapsto (F(x), \eta_\sigma \circ F(\theta_\sigma))$ yields a dg functor $F^\cG\colon \cA^\cG \to \cB^\cG$. We hence obtain a functor
\begin{eqnarray}\label{eq:equiv-functor}
\dgcat^\cG(k) \too \dgcat(k) && \cG \circlearrowright \cA \mapsto \cA^\cG\,.
\end{eqnarray}
\begin{proposition}\label{prop:adjunction}
We have the adjunction of categories:
\begin{equation*}\label{eq:adjunction-19}
\xymatrix{
\dgcat^\cG(k) \ar@<1ex>[d]^-{\eqref{eq:equiv-functor}}\\
\dgcat(k) \ar@<1ex>[u]^-{\eqref{eq:trivial}}\,.
}
\end{equation*}
\end{proposition}
\begin{proof}
Let $\cA$ be a small dg category and $\cB$ a small $\cG$-dg category. The unit of the adjunction is given by the dg functors $\cA \to \cA^\cG, x \mapsto (x, \theta_\sigma:= \id)$, and the counit by the $\cG$-equivariant dg functors $ \cG \circlearrowright_1 \cB^\cG \to \cG \circlearrowright \cB, (x, \theta_\sigma) \mapsto ((x, \theta_\sigma)\mapsto x, \eta_\sigma:=\theta_\sigma)$. This data satisfies the axioms of an adjunction.
%
%
\end{proof}
\begin{remark}[$\cG^\vee$-action]\label{rk:character-action}
Let $\cA$ be a $\cG$-dg category and $\cA^\cG$ the associated dg category of $\cG$-equivariant objects. Given a character $\chi\colon \cG \to k^\times$, the assignment $(x, \theta_\sigma) \mapsto (x, \chi(\sigma)\theta_\sigma)$ yields a dg equivalence $\phi_\chi\colon \cA^\cG \to \cA^\cG$. These dg equivalences and the natural isomorphisms of dg functors $\epsilon_{\psi,\chi}:=\id$ equip $\cA^\cG$ with a $\widehat{\cG}$-action, where $\widehat{\cG}$ stands for the group of characters of $\cG$. Since this construction is functorial on $\cA$, it gives rise to a functor
\begin{eqnarray}\label{eq:lifting}
\dgcat^\cG(k) \too \dgcat^{\cG^\vee}(k) && \cG \circlearrowright \cA \mapsto \cG^\vee \circlearrowright \cA^\cG\,.
\end{eqnarray}
The composition of \eqref{eq:lifting} with the restriction functor \eqref{eq:forgetful} agrees with \eqref{eq:equiv-functor}.
\end{remark}
\subsection{Twisted equivariant objects}
Let $\alpha\colon \cG \times \cG \to k^\times$ be a $2$-cocycle and $\cG \circlearrowright \cA$ a $\cG$-dg category. Similarly to Definition \ref{def:G-equivariant}, an {\em $\alpha$-twisted $\cG$-equivariant object in $\cG \circlearrowright \cA$} consists of an object $x \in \cA$ and of a family of closed degree zero isomorphisms $\theta_\sigma\colon x \to \phi_\sigma(x), \sigma \in \cG$, with $\theta_1=\id$, such that the compositions 
$$
x \stackrel{\theta_\rho}{\too} \phi_\rho(x) \stackrel{\phi_\rho(\theta_\sigma)}{\too} \phi_\rho(\phi_\sigma (x)) \stackrel{\epsilon_{\rho,\sigma}(x)}{\too} \phi_{\rho\sigma}(x)
$$
are equal to $\alpha( \rho, \sigma) \theta_{\rho\sigma}\colon x \to \phi_{\rho\sigma}(x)$ for every $\sigma, \rho \in \cG$. A {\em morphism of $\alpha$-twisted $\cG$-equivariant objects $(x,\theta_\sigma) \to (y,\theta_\sigma)$} is an element $f$ of the dg $k$-vector space $\cA(x,y)$ such that $\theta_\sigma \circ f = \phi_\sigma(f) \circ \theta_\sigma$ for every $\sigma \in \cG$. Let us write $\cA^{\cG,\alpha}$ for the dg category of $\alpha$-twisted $\cG$-equivariant objects in $\cG \circlearrowright \cA$. Note that $\cA^{\cG, \alpha}$ identifies with the dg category of $\cG$-equivariant objects in $(\cG \circlearrowright \cA) \otimes (\cG \circlearrowright_{\alpha^{-1}}k)$.
\begin{example}[Twisted equivariant perfect complexes]\label{ex:alpha-twisted}
Let $\cG \circlearrowright \perf_\dg(X)$ be as in Example \ref{ex:G-schemes}. Similarly to Example \ref{ex:G-equivariant}, $\perf_\dg(X)^{\cG, \alpha}$ is Morita equivalent to the dg category of $\alpha$-twisted $\cG$-equivariant perfect complexes $\perf_\dg^{\cG, \alpha}(X)$.
\end{example}
\subsection{Group actions on categories}
All the constructions and results of \S\ref{sec:actions} hold {\em mutatis mutandis} for ordinary categories: simply remove the shorthand ``dg''. This fact was already implicitly used in the proof of Lemma \ref{lem:generator}. 
%
\section{Equivariant noncommutative motives}\label{sec:equivariant}
In this section we introduce the categories of equivariant noncommutative Chow motives and equivariant noncommutative numerical motives. We start by recalling the definition of their non-equivariant predecessors; for further information on noncommutative motives, we invite the reader to consult the book \cite{book}. In the remainder of the article we will always assume that $\mathrm{char}(k)\nmid |\cG|$. 
\subsection{Noncommutative Chow motives}
As proved in \cite[Cor.~5.10]{Additive}, there is a canonical bijection between $\Hom_{\Hmo(k)}(\cA,\cB)$ and the set of isomorphism classes of the triangulated category $\rep(\cA,\cB)$. Under this bijection, the composition law of $\Hmo(k)$ is induced by the triangulated bifunctors
\begin{eqnarray}\label{eq:bifunctor}
\rep(\cA,\cB) \times \rep(\cB,\cC) \too \rep(\cA,\cC) && (\mathrm{B}, \mathrm{B}') \mapsto \mathrm{B}\otimes_\cB \mathrm{B}'
\end{eqnarray}
and the localization functor from $\dgcat(k)$ to $\Hmo(k)$ is given by 
\begin{eqnarray}\label{eq:functor1}
\dgcat(k) \too \Hmo(k) & \cA \mapsto \cA & (\cA \stackrel{F}{\to}\cB) \mapsto {}_F \mathrm{B}\,.
\end{eqnarray}
The {\em additivization} of $\Hmo(k)$ is the additive category $\Hmo_0(k)$ with the same objects and with abelian groups of morphisms $\Hom_{\Hmo_0(k)}(\cA,\cB)$ given by $K_0\rep(\cA,\cB)$, where $K_0\rep(\cA,\cB)$ stands for the Grothendieck group of the triangulated category $\rep(\cA,\cB)$. The composition law is induced by the triangulated bifunctors \eqref{eq:bifunctor}. By construction, $\Hmo_0(k)$ comes equipped with the functor
\begin{eqnarray}\label{eq:functor2}
\Hmo(k) \too \Hmo_0(k) & \cA \mapsto \cA & \mathrm{B} \mapsto [\mathrm{B}]\,.
\end{eqnarray}
Let us denote by $U\colon \dgcat(k) \to \Hmo_0(k)$ the composition $\eqref{eq:functor2}\circ \eqref{eq:functor1}$. As proved in \cite[Lem.~6.1]{Additive}, the category $\Hmo_0(k)$ carries a symmetric monoidal structure induced by the tensor product of dg categories and by the triangulated bifunctors 
\begin{eqnarray*}
\rep(\cA,\cB) \times \rep(\cC,\cD) \too \rep(\cA\otimes\cC,\cB\otimes \cD) && (\mathrm{B}, \mathrm{B}') \mapsto \mathrm{B}\otimes \mathrm{B}'\,.
\end{eqnarray*}
By construction, the functor $U$ is symmetric monoidal. The category $\NChow(k)$ of {\em noncommutative Chow motives} is defined as the idempotent completion of the full subcategory of $\Hmo_0(k)$ consisting of the objects $U(\cA)$ with $\cA$ a smooth proper dg category. The category $\NChow(k)$ is additive, idempotent complete, and rigid symmetric monoidal (\ie all its objects are strongly dualizable).
\subsection{Noncommutative numerical motives}\label{sub:NNmotives}
Given an additive rigid symmetric monoidal category $\cC$, its {\em $\cN$-ideal} is defined as follows
$$ \cN(a,b)=\{f \in \Hom_\cC(a,b)\,|\,\forall g \in \Hom_\cC(b,a)\,\,\mathrm{we}\,\,\mathrm{have}\,\,\mathrm{tr}(g\circ f)=0 \}\,,$$
where $\mathrm{tr}(g\circ f)$ stands for the categorical trace of the endomorphism $g \circ f$. The category $\NNum(k)$ of {\em noncommutative numerical motives} is defined as the idempotent completion of the quotient $\NChow(k)/\cN$.
\subsection{Equivariant noncommutative Chow motives}
Let $\cG \circlearrowright \cA$ and $\cG \circlearrowright \cB$ be two small $\cG$-dg categories. As mentioned in Remark \ref{rk:dgbimodules}, the dg category $\rep_\dg(\cA,\cB)$ inherits a $\cG$-action. As a consequence, we obtain an induced $\cG$-action on the triangulated category $\dgHo(\rep_\dg(\cA,\cB)) \simeq \rep(\cA,\cB)$. Thanks to \cite[Thm.~8.7]{Elagin}, the  category of $\cG$-equivariant objects $\rep(\cA,\cB)^\cG$ is also triangulated.

Given small $\cG$-dg categories $\cG \circlearrowright \cA$, $\cG \circlearrowright \cB$, and $\cG \circlearrowright \cC$, consider the following $\cG$-equivariant dg functor ($\cG$ acts diagonally on the left-hand side)
\begin{eqnarray*}
\rep_\dg(\cA, \cB) \times \rep_\dg(\cB,\cC) \too \rep_\dg(\cA,\cC) && (\mathrm{B},\mathrm{B}') \mapsto \mathrm{B}\otimes_\cB \mathrm{B}'\,.
\end{eqnarray*}
By first applying $\dgHo(-)$ and then $(-)^\cG$, we obtain an induced triangulated bifunctor
\begin{equation}\label{eq:bifunctor1}
\rep(\cA,\cB)^\cG \times \rep(\cB,\cC)^\cG \too \rep(\cA,\cC)^\cG\,.
\end{equation}
Let $\Hmo^\cG(k)$ be the category with the same objects as $\dgcat^\cG(k)$ and with morphisms $\Hom_{\Hmo^\cG(k)}(\cG \circlearrowright \cA, \cG \circlearrowright \cB)$ given by the set of isomorphism classes of the category $\rep(\cA,\cB)^\cG$. The composition law is induced by the triangulated bifunctors \eqref{eq:bifunctor1}. Thanks to Remark \ref{rk:bijection}, we have the functor:
\begin{equation}\label{eq:functor1-1}
\dgcat^\cG(k) \too \Hmo^\cG(k) \quad \cG \circlearrowright\cA \mapsto \cG \circlearrowright \cA \quad (\cG \circlearrowright\cA \stackrel{F}{\to} \cG \circlearrowright \cB) \mapsto {}_F \mathrm{B}\,.
\end{equation}
\begin{lemma}\label{lem:inverts}
The functor \eqref{eq:functor1-1} inverts $\cG$-equivariant Morita equivalences.
\end{lemma}
\begin{proof}
Let $\cG \circlearrowright \cA \to \cG \circlearrowright\cB$ be a $\cG$-equivariant Morita equivalence. Thanks to the Yoneda lemma, it suffices to show that for every object $\cG \circlearrowright \cC$ the homomorphism
\begin{equation}\label{eq:induced-new}
\Hom_{\Hmo^\cG(k)}(\cG \circlearrowright \cC, \cG \circlearrowright \cA) \too \Hom_{\Hmo^\cG(k)}(\cG \circlearrowright \cC, \cG \circlearrowright \cB)
\end{equation}
is invertible. Since $\cG \circlearrowright \cA \to \cG \circlearrowright\cB$ is a $\cG$-equivariant Morita equivalence, we have an induced $\cG$-equivariant equivalence of categories $\rep(\cC, \cA) \to \rep(\cC,\cB)$, and consequence an equivalence of categories $\rep(\cC,\cA)^\cG \to \rep(\cC,\cB)^\cG$.
\end{proof}
The {\em additivization} of $\Hmo^\cG(k)$ is the category $\Hmo_0^\cG(k)$ with the same objects and with abelian groups of morphisms $\Hom_{\Hmo_0^\cG(k)}(\cG \circlearrowright \cA, \cG \circlearrowright \cB)$ given by $K_0\rep(\cA,\cB)^\cG$, where $K_0\rep(\cA,\cB)^\cG$ stands for the Grothendieck group of the triangulated category $\rep(\cA,\cB)^\cG$. The composition law is induced by the triangulated bifunctors \eqref{eq:bifunctor1}. By constrution, $\Hmo_0^\cG(k)$ comes equipped with the functor
\begin{eqnarray}\label{eq:functor2-2}
\Hmo^\cG(k) \too \Hmo_0^\cG(k) & \cG \circlearrowright\cA \mapsto \cG \circlearrowright \cA & \mathrm{B} \mapsto [\mathrm{B}]\,.
\end{eqnarray}
Let us denote by $U^\cG\colon \dgcat^\cG(k) \to \Hmo^\cG_0(k)$ the composition $\eqref{eq:functor2-2}\circ\eqref{eq:functor1-1}$. 

Given small $\cG$-dg categories $\cG \circlearrowright \cA$, $\cG \circlearrowright \cB$, $\cG \circlearrowright \cC$, and $\cG \circlearrowright \cD$, consider the following $\cG$-equivariant dg functor ($\cG$ acts diagonally on the left-hand side) 
\begin{eqnarray*}
\rep_\dg(\cA,\cB) \times \rep_\dg(\cC,\cD) \too \rep_\dg(\cA\otimes\cC,\cB\otimes \cD) && (\mathrm{B}, \mathrm{B}') \mapsto \mathrm{B}\otimes \mathrm{B}'\,.
\end{eqnarray*}
By first applying $\dgHo(-)$ and then $(-)^\cG$, we obtain an induced triangulated bifunctor
\begin{equation}\label{eq:pairing-monoidal}
\rep(\cA,\cB)^\cG \times \rep(\cC,\cD)^\cG \too \rep(\cA\otimes \cC,\cB \otimes \cD)^\cG \,.
\end{equation}
The assignment $(\cG \circlearrowright \cA, \cG \circlearrowright \cB) \mapsto \cG \circlearrowright(\cA \otimes \cB)$, combined with the triangulated bifunctors \eqref{eq:pairing-monoidal}, gives rise to a symmetric monoidal structure on $\Hmo^\cG_0(k)$ with $\otimes$-unit $U^\cG(\cG \circlearrowright_1k)$. By construction, the functor $U^\cG$ is symmetric monoidal. 
\begin{proposition}
The category $\Hmo_0^\cG(k)$ is additive. Moreover, we have
\begin{equation}\label{eq:identification}
U^\cG(\cG \circlearrowright \cA) \oplus U^\cG(\cG \circlearrowright \cB) \simeq U^\cG(\cG \circlearrowright (\cA \times \cB)) \simeq U^\cG(\cG \circlearrowright(\cA \amalg \cB))
\end{equation}
for any two small $\cG$-dg categories $\cG \circlearrowright \cA$ and $\cG \circlearrowright \cB$.
\end{proposition}
\begin{proof}
By construction, the morphism sets of $\Hmo_0^\cG(k)$ are abelian groups and the composition law is bilinear. Hence, it suffices to show the isomorphisms \eqref{eq:identification}, which imply in particular that the category $\Hmo_0^\cG(k)$ admits direct sums. Given a small $\cG$-dg category $\cG \circlearrowright \cC$, we have equivalences of categories
$$
\rep(\cC,\cA\times \cB)^\cG \simeq \rep(\cC,\cA)^\cG \times \rep(\cC,\cB)^\cG \,\,\,\,
\rep(\cA \amalg \cB,\cC)^\cG \simeq \rep(\cA,\cC)^\cG \times \rep(\cB,\cC)^\cG
$$
By passing to the Grothendieck group $K_0$, we conclude that $U^\cG(\cG \circlearrowright (\cA \times \cB))$, resp. $U^\cG(\cG \circlearrowright (\cA \amalg \cB))$, is the product, resp. coproduct, in $\Hmo_0^\cG(k)$ of $U^\cG(\cG \circlearrowright \cA)$ with $U^\cG(\cG \circlearrowright \cB)$. Using the fact that the category $\Hmo_0^\cG(k)$ is $\bbZ$-linear, we  obtain finally the isomorphisms \eqref{eq:identification}.
\end{proof}
%
%
%
\begin{definition}
The category $\NChow^\cG(k)$ of {\em $\cG$-equivariant noncommutative Chow motives} is the idempotent completion of the full subcategory of $\Hmo_0^\cG(k)$ consisting of the objects $U^\cG(\cG \circlearrowright \cA)$ with $\cA$ a smooth proper dg category.
\end{definition}
Since the smooth proper dg categories are stable under (co)products, it follows from the isomorphisms \eqref{eq:identification} that the category $\NChow^\cG(k)$ is also additive.

\begin{proposition}\label{prop:rigid}
The symmetric monoidal category $\NChow^\cG(k)$ is rigid.
\end{proposition}
\begin{proof}
By construction of $\NChow^\cG(k)$, it suffices to show that $U^\cG(\cG \circlearrowright \cA)$, with $\cA$ a smooth proper dg category $\cA$, is strongly dualizable. Take for dual of $U^\cG(\cG \circlearrowright \cA)$ the object $U^\cG(\cG \circlearrowright \cA^\op)$ (see Remark \ref{rk:opposite}). The dg $\cA\text{-}\cA$-bimodule
\begin{eqnarray}\label{eq:bim}
{}_{\id}\mathrm{B} \colon \cA\otimes \cA^\op \too \cC_\dg(k) && (x,y) \mapsto \cA(y,x)\,,
\end{eqnarray}
associated to the identity dg functor $\id\colon \cA \to \cA$, is canonically a $\cG$-equivariant object. Moreover, since $\cA$ is smooth proper, the dg $\cA\text{-}\cA$-bimodule \eqref{eq:bim} belongs to the triangulated categories $\rep(\cA\otimes \cA^\op, k)^\cG$ and $\rep(k, \cA^\op \otimes \cA)^\cG$. Let us then take for the evaluation morphism the Grothendieck class of \eqref{eq:bim} in
$$\Hom_{\NChow^\cG(k)}(U^\cG(\cG \circlearrowright (\cA \otimes \cA^\op)),U^\cG(\cG \circlearrowright_1 k)) \simeq K_0 \rep(\cA\otimes \cA^\op, k)^\cG\,,$$
and for the coevaluation morphism the Grothendieck class of \eqref{eq:bim} in
$$\Hom_{\NChow^\cG(k)}(U^\cG(\cG \circlearrowright_1 k), U^\cG(\cG \circlearrowright (\cA^\op \otimes \cA))) \simeq K_0 \rep(k, \cA^\op\otimes \cA)^\cG\,.$$
This data satisfies the axioms of a strongly dualizable object. 
%
\end{proof}
\begin{proposition}\label{prop:endomorphisms}
For every cohomology class $[\alpha] \in H^2(\cG,k^\times)$, the ring of endomorphisms (where multiplication is given by composition)
\begin{equation}\label{eq:endomorphisms-1}
\End_{\NChow^\cG(k)}(U^\cG(\cG \circlearrowright_\alpha k))
\end{equation}
is isomorphic to the representation ring $R(\cG)$ of the group $\cG$.
\end{proposition}
\begin{proof}
By construction of $\NChow^\cG(k)$, we have canonical ring identifications
$$ \End(U^\cG(\cG \circlearrowright_\alpha k)) = K_0(\rep(k,k)^{\cG, \alpha\alpha^{-1}})\simeq K_0\rep(k,k)^\cG = \End(U^\cG(\cG \circlearrowright_1 k))\,.$$
Hence, it suffices to prove the particular case $\alpha=1$. As mentioned in Example  \ref{ex:G-equivariant}, the category $\rep(k,k)^\cG \simeq \cD_c(k)^\cG \simeq \perf(\mathrm{Spec}(k))^\cG$ is equivalent to $\perf^\cG(\mathrm{Spec}(k))$. This implies that the abelian group \eqref{eq:endomorphisms-1}, with $\alpha=1$, is isomorphic to the $\cG$-equivariant Grothendieck group $K_0(\perf^\cG(\mathrm{Spec}(k)))$ of $\mathrm{Spec}(k)$. In what concerns the ring structure, the Eckmann-Hilton argument, combined with the fact that $U^\cG(\cG \circlearrowright_1 k)$ is the $\otimes$-unit of $\NChow^\cG(k)$, implies that the multiplication on \eqref{eq:endomorphisms-1} given by composition agrees with the multiplication on \eqref{eq:endomorphisms-1} induced by the symmetric monoidal structure on $\perf^\cG(\mathrm{Spec}(k))$. The proof follows now from the definition of $R(\cG)$ as the $\cG$-equivariant Grothendieck ring of $\mathrm{Spec}(k)$.
\end{proof}

\begin{example}
\begin{itemize}
\item[(i)]
When $k=\bbC$ and $\cG$ is abelian, the representation ring $R(\cG)$ identifies with the group ring $\bbZ[\widehat{\cG}]$. For example, when $\cG=C_n$ is the cyclic group of order $n$, we have $R(C_n)\simeq \bbZ[\chi]/\langle \chi^n -1\rangle$;
\item[(ii)] When $k=\bbC$ and $\cG=S_3$ is the symmetric group on $3$ letters, we have $R(S_3)\simeq \langle 1, \chi, \psi\,|\, \chi\psi=\psi\chi, \chi^2=1, \psi^2=1+\chi+\psi \rangle$;
\item[(iii)] When $k=\bbQ$ and $\cG=C_3$, we have $R(C_3)\simeq \bbZ[\chi]\langle \chi^2- \chi -2\rangle$;
\item[(iv)] Consult Serre's book \cite{book} for a detailed study of the representation ring.
\end{itemize}
\end{example}
Proposition \ref{prop:endomorphisms} gives automatically rise to the following enhancement:
\begin{corollary}\label{cor:R(G)-linear}
The category $\NChow^\cG(k)$ (and $\Hmo_0^\cG(k)$) is $R(\cG)$-linear.
\end{corollary}
\subsubsection{Restriction/Induction}
Recall from \eqref{eq:forgetful} the definition of the restriction functor. Clearly, it gives rise to an additive functor
\begin{eqnarray}\label{eq:restriction-1}
\Hmo_0^\cG(k) \too \Hmo_0(k) && U^\cG(\cG \circlearrowright \cA) \mapsto U(\cA)\,.
\end{eqnarray}
In the converse direction, we have the induction functor 
\begin{eqnarray}\label{eq:induction}
\Hmo_0(k) \too \Hmo_0^\cG(k) && U(\cA) \mapsto U^\cG(\cG \circlearrowright \amalg_{\rho \in \cG} \cA)\,,
\end{eqnarray}
where $\cG$ acts by permutation of the components.
\begin{proposition}\label{prop:adjunctions}
We have adjunctions of categories:
\begin{equation}\label{eq:3-adjunctions}
\xymatrix{
\Hmo_0^\cG(k) \ar@<1ex>[d]^-{\eqref{eq:restriction-1}} && \NChow^\cG(k) \ar@<1ex>[d]^-{\eqref{eq:restriction-1}} \\
\Hmo_0(k) \ar@<1ex>[u]^-{\eqref{eq:induction}} && \NChow(k) \ar@<1ex>[u]^-{\eqref{eq:induction}}\,.
}
\end{equation}
\end{proposition}
\begin{proof}
Given a small dg category $\cA$ and a small $\cG$-dg category $\cG \circlearrowright \cB$, we have a natural equivalence of triangulated categories
\begin{eqnarray*}
\rep(\amalg_{\rho \in \cG} \cA, \cB)^\cG \simeq (\Pi_{\rho \in \cG} \rep(\cA,\cB))^\cG \too \rep(\cA,\cB) && (\{\mathrm{B}_\rho\}_{\rho \in \cG}, \theta_\sigma)\mapsto \mathrm{B}_1\,.
\end{eqnarray*}
By passing to the Grothendieck group $K_0$, we obtain the left-hand side adjunction. The right-hand side adjunction follows now from the fact that smooth proper dg categories are stable under coproducts.
\end{proof}
\subsection{Equivariant noncommutative numerical motives}\label{sub:NCnumerical}
The category of $\cG$-equivariant noncommutative Chow motives is additive and rigid symmetric monoidal. Therefore, following \S\ref{sub:NNmotives}, the category $\NNum^\cG(k)$ of {\em $\cG$-equivariant noncommutative numerical motives} is defined as the idempotent completion of the quotient $\NChow^\cG(k)/\cN$. Since $\cN$ is a $\otimes$-ideal, the category $\NNum^\cG(k)$ is not only additive and idempotent complete, but also rigid symmetric monoidal. Note that endomorphism ring $\End_{\NNum^\cG(k)}(U^\cG(\cG \circlearrowright_\alpha k))$ is isomorphic to \eqref{eq:endomorphisms-1}. Similarly to Corollary \ref{cor:R(G)-linear}, this implies that $\NNum^\cG(k)$ is moreover $R(\cG)$-linear.
\subsubsection{Bilinear form} Let $\cG \circlearrowright \cA$ be a small $\cG$-dg category with $\cA$ proper. Given $\cG$-equivariant objects $(M,\theta_\sigma), (N,\theta_\sigma) \in \cD_c(\cA)^\cG$, the $k$-vector space $\Hom_{\cD_c(\cA)}(M,N)$ is finite dimensional and comes equipped with the $\cG$-action $(\sigma, f) \mapsto \theta_\sigma^{-1} \circ \phi_\sigma(f) \circ \theta_\sigma$. As a consequence, we obtain an induced bilinear form
$$
\langle-,- \rangle^\cG \colon K_0(\cD_c(\cA)^\cG) \otimes K_0(\cD_c(\cA)^\cG) \too R(\cG)
$$
defined as $([(M,\theta_\sigma)],[(N,\theta_\sigma)])\mapsto \sum_i (-1)^i \Hom_{\cD_c(\cA)}(M,N[i])$. The next result describes the $\otimes$-ideal $\cN$ in terms of the preceding bilinear form. Since this result is not used in the article, we leave the proof to the reader.
\begin{proposition}
Let $\cG \circlearrowright \cA$ and $\cG \circlearrowright \cB$ be two $\cG$-dg categories with $\cA$ and $\cB$ smooth proper. Given $[(M,\theta_\sigma)] \in K_0\rep(\cA,\cB)^\cG$, the conditions are equivalent:
\begin{itemize}
\item[(i)] We have $[(M,\theta_\sigma)]=0$ in $\Hom_{\NNum^\cG(k)}(U^\cG(\cG \circlearrowright \cA),U^\cG(\cG \circlearrowright \cB))$;
\item[(ii)] We have $\langle [(M,\theta_\sigma)], [(N, \theta_\sigma)]\rangle^\cG=0$ for every $[(N,\theta_\sigma)] \in K_0\rep(\cA,\cB)^\cG$\,.
\end{itemize}
\end{proposition}
\subsection{Coefficients}\label{sub:coefficients}
Given a commutative ring $R$, let us write $\Hmo_0^\cG(k)_R$ for the $R$-linear additive category obtained from $\Hmo_0^\cG(k)$ by tensoring each abelian group of morphisms with $R$. By construction, $\Hmo_0^\cG(k)_R$ inherits from $\Hmo_0^\cG(k)$ a symmetric monoidal structure making the functor $ (-)_R\colon \Hmo_0^\cG(k) \to \Hmo_0^\cG(k)_R$ symmetric monoidal.  The category $\NChow^\cG(k)_R$ of {\em $\cG$-equivariant noncommutative Chow motives with $R$-coefficients} is defined as the idempotent completion of the subcategory of $\Hmo_0^\cG(k)_R$ consisting of the objects $U^\cG(\cG \circlearrowright \cA)_R$ with $\cA$ a smooth proper dg category. In the same vein, the category $\NNum^\cG(k)_R$ of {\em $\cG$-equivariant noncommutative numerical motives with $R$-coefficients} is defined as the idempotent completion of the quotient $\NChow^\cG(k)_R/\cN$. When $\bbQ \subseteq R$, $\NNum^\cG(k)_R$ is equivalent to the idempotent completion of the category obtained from $\NNum^\cG(k)$ by tensoring each abelian group of morphisms with $R$; see \cite[Prop.~1.4.1]{Brugieres}.
\subsection{Decomposition}\label{sub:decomposition}
\begin{proposition}\label{prop:decomposition}
When $1/|\cG|\in R$, we have a decomposition of $R$-linear, idempotent complete, rigid symmetric monoidal categories:
\begin{equation}\label{eq:decomposition}
\NChow^\cG(k)_R\simeq \NChow(k)_R \times \NChow^\cG(k)^-_R\,.
\end{equation}
Under this decomposition, the restriction functor  $\NChow^\cG(k)_R \to \NChow(k)_R$ corresponds to the projection onto $\NChow(k)_R$. Consequently, the same holds for $\cG$-equivariant noncommutative numerical motives.
\end{proposition}
\begin{proof}
Let us denote by $k\cG$ the regular representation. Since $1/|\cG| \in R$, we can consider the orthogonal idempotents~$e^+:=[k\cG]/|\cG|$~and~$e^-:=1 - [k\cG]/|\cG|$~of the representation ring $R(\cG)_R$. Using the identification of Proposition \ref{prop:endomorphisms} between $\End_{\NChow^\cG(k)_R}(U^\cG(\cG \circlearrowright_1k)_R)$ and $R(\cG)_R$ and the fact that $\NChow^\cG(k)_R$ is idempotent complete,~we~obtain~a~decomposition~of~the~$\otimes$-unit~object 
$$U^\cG(\cG \circlearrowright_1k)_R \simeq U^\cG(\cG \circlearrowright_1k)^+_R \oplus U^\cG(\cG \circlearrowright_1k)^-_R\,,$$ 
where $U^\cG(\cG \circlearrowright_1k)^+_R$, resp. $U^\cG(\cG \circlearrowright_1k)^-_R$, stands for the image of $e^+$, resp. $e^-$. Given an object $N\!\!M\in \NChow^\cG(k)$, let $N\!\!M^+$, resp. $N\!\!M^-$, be the tensor product of $N\!\!M$ with $U^\cG(\cG \circlearrowright_1 k)^+_R$, resp. $U^\cG(\cG \circlearrowright_1 k)^-_R$. In the same vein, let us write $\NChow^\cG(k)^+_R$, resp. $\NChow^\cG(k)^-_R$, for the full subcategory of $\NChow^\cG(k)_R$ consisting of the objects $N\!\!M^+$, resp. $N\!\!M^-$. Since $e^+$ and $e^-$ are orthogonal idempotents, the direct sum functor $(N\!\!M^+, N\!\!M^-) \mapsto N\!\!M^+ \oplus N\!\!M^-$ gives rise to a decomposition of $R$-linear, idempotent complete, rigid symmetric monoidal categories:
$$\NChow^\cG(k)_R\simeq \NChow^\cG(k)^+_R \times \NChow^\cG(k)^-_R\,.$$
The rank homomorphism $R(\cG)_R \twoheadrightarrow R$ sends $e^+$ to $1$ and $e^-$ to $0$. Consequently, given a small $\cG$-dg category $\cG \circlearrowright \cA$, the restriction functor sends $U^\cG(\cG \circlearrowright \cA)^+_R$ to $U(\cA)_R$ and $U^\cG(\cG \circlearrowright \cA)^-_R$ to zero. Since this functor is essentially surjective, it remains then only to show that the restriction homomorphisms
$$ \Hom_{\NChow^\cG(k)_R}(U^\cG(\cG \circlearrowright \cA)^+_R, U^\cG(\cG \circlearrowright \cB)^-_R)\too \Hom_{\NChow(k)_R}(U(\cA)_R, U(\cB)_R)$$
are invertible. Their inverses are provided by the homomorphisms
\begin{eqnarray*}
K_0(\rep(\cA,\cB))_R \too K_0(\rep(\cA,\cB)^\cG)_R && [M] \mapsto [(\oplus_{\rho \in \cG}\phi_\rho(M), \theta_\sigma)]/|\cG|\,,
\end{eqnarray*}
where $\theta_\sigma$ is given by the collection of isomorphisms $\epsilon_{\sigma, \rho}(M)^{-1}$.
\end{proof}
\section{Equivariant and enhanced additive invariants}\label{sec:additive}
Given a small dg category $\cA$, let $T(\cA)$ be the dg category of pairs $(i,x)$, where $i \in \{1,2\}$ and $x \in \cA$. The dg $k$-vector spaces $T(\cA)((i,x), (j,y))$ are given by $\cA(x,y)$ if $j \geq i$ and are zero otherwise. Note that we have two inclusion dg functors $\iota_1, \iota_2\colon \cA \to T(\cA)$. A functor $E\colon \dgcat(k) \to \mathrm{D}$, with values in an additive category, is called an {\em additive invariant} if it satisfies the following two conditions:
\begin{itemize}
\item[(i)] it sends Morita equivalences to isomorphisms;
\item[(ii)] given a small dg category $\cA$, the dg functors $\iota_1, \iota_2$ induce an isomorphism\footnote{Condition (ii) can be equivalently formulated in terms of semi-orthogonal decompositions in the sense of Bondal-Orlov \cite{BO}; consult \cite[Thm.~6.3(4)]{Additive} for details.}
$$ (E(\iota_1)\,\,E(\iota_2)) \colon E(\cA) \oplus E(\cA) \too E(T(\cA))\,.$$
\end{itemize}
Examples of additive invariants include algebraic $K$-theory, Hochschild homology $HH$, cyclic homology $HC$, periodic cyclic homology $HP$, negative cyclic homology $HN$, the mixed complex $\mathrm{C}$, topological Hochschild homology $THH$, topological cyclic homology $TC$, etc; consult \cite[\S2.2]{book} for details. As proved in \cite[Thms.~5.3 and 6.3]{Additive}, the functor $U\colon \dgcat(k) \to \Hmo_0(k)$ is the {\em universal additive invariant}, \ie given any additive category $\mathrm{D}$ we have an induced equivalence of categories
\begin{equation}\label{eq:equivalence-categories}
U^\ast\colon \Fun_{\mathrm{additive}}(\Hmo_0(k), \mathrm{D}) \too \Fun_{\mathrm{add}}(\dgcat(k), \mathrm{D})\,,
\end{equation}
where the left-hand side denotes the category of additive functors and the right-hand side the category of additive invariants. 
\begin{remark}[Additive invariants of twisted group algebras]\label{rk:additivetwisted}
Let $\alpha\colon \cG \times \cG \to k^\times$ be a $2$-cocycle and $k_\alpha[\cG]$ the associated twisted group algebra. Recall that a conjugacy class $\langle \sigma \rangle$ of $\cG$ is called {\em $\alpha$-regular} if $\alpha(\sigma,\rho)=\alpha(\rho,\sigma)$ for every element $\rho$ of the centralizer $C_\cG(\sigma)$. Thanks to the (generalized) Maschke theorem, the algebra $k_\alpha[\cG]$ is semi-simple. Moreover, the number of simple $k_\alpha[\cG]$-modules is equal to the number $|\langle \cG\rangle_\alpha |$ of $\alpha$-regular conjugacy classes of $\cG$. Let $E\colon \dgcat(k) \to \mathrm{D}$ be an additive invariant. Making use of \cite[Cor.~3.20 and Rk.~3.21]{Azumaya}, we obtain~the~computations:
\begin{itemize}
\item[(i)] We have $E(k_\alpha[\cG]) \simeq \oplus^{|\langle \cG\rangle_\alpha|}_{i=1} E(D_i)$, where $D_i:=\End_{k_\alpha[\cG]}(S_i)$ is the division algebra associated to the simple (right) $k_\alpha[\cG]$-module $S_i$;
\item[(ii)] When $\mathrm{D}$ is $\bbQ$-linear, we have $E(k_\alpha[\cG])\simeq \oplus^{|\langle \cG\rangle_\alpha|}_{i=1} E(l_i)$ where $l_i$ (a finite field extension of $k$) is the center of $D_i$;
\item[(iii)] When $k$ is algebraically closed, we have $E(k_\alpha[\cG])\simeq E(k)^{\oplus |\langle \cG\rangle_\alpha|}$.
\end{itemize}
\end{remark}
\subsection{Equivariant additive invariants}\label{sub:additive}
Given an additive invariant $E$, the associated {\em $\cG$-equivariant additive invariant} is defined as the composition
\begin{equation}\label{eq:composed-G}
E^\cG \colon \dgcat^\cG(k) \stackrel{\eqref{eq:equiv-functor}}{\too} \dgcat(k) \stackrel{E}{\too} \mathrm{D}\,.
\end{equation}
From a topological viewpoint, $E^\cG(\cG \circlearrowright\cA)$ may be understood as the value of $E$ at the ``homotopy fixed points'' of the $\cG$-action on $\cA$. Here are some examples:
\begin{example}\label{ex:twisted-new}
\begin{itemize}
\item[(i)] Let $\cG \circlearrowright \perf_\dg(X)$ be as in Example~\ref{ex:G-schemes}. Thanks to Example \ref{ex:G-equivariant}, we have an identification between $E^\cG(\cG \circlearrowright \perf_\dg(X))$ and $E(\perf_\dg^\cG(X))$;
\item[(ii)] Let $\cG \circlearrowright \perf_\dg(X)$ be as in Example~\ref{ex:lines}. Thanks to Example \ref{ex:covering}, we  have an identification between $E^\cG(\cG \circlearrowright \perf_\dg(X))$ and $E(\perf_\dg(Y))$;
\item[(iii)] Let $\cG \circlearrowright A$ be as in Example \ref{ex:G-algebras}. Thanks to Example \ref{ex:semidirect}, we have an identification between $E^\cG(\cG \circlearrowright \cC_{c, \dg}(A))$ and $E(A \rtimes \cG)$;
\item[(iv)] Let $\cG \circlearrowright_\alpha k$ be as in Example \ref{ex:2-cocycles}. Thanks to Example \ref{ex:twisted}, we have an identification between $E^\cG(\cG \circlearrowright_\alpha \cC_{c,\dg}(k))$ and $E(k_\alpha[\cG])$. 
\end{itemize}
\end{example}
\begin{example}[Equivariant algebraic $K$-theory]
The composed functor \eqref{eq:composed-G} with $E:=K$ is called {\em $\cG$-equivariant algebraic $K$-theory}. Recall that a quasi-compact quasi-separated $\cG$-scheme $X$ has the {\em resolution property} if every $\cG$-equivariant coherent $\cO_X$-module is a quotient of a $\cG$-bundle. For example, the existence of an ample family of line $\cG$-bundles implies the resolution property. As explained in \cite[Cor.~3.9]{Khrisna}, whenever $X$ has the resolution property, $K^\cG(\cG \circlearrowright \perf_\dg(X))\simeq K(\perf_\dg^\cG(X))$ agrees with the $\cG$-equivariant algebraic $K$-theory $K^\cG(X)$ of $X$ in the sense of Thomason \cite[\S1.4]{Thomason}. 
\end{example}
\begin{example}[Equivariant Hochschild, cyclic, periodic, and negative homology]
The composed functor \eqref{eq:composed-G} with $E:=HH, HC, HP$, and $HN$, is called {\em $\cG$-equivariant Hochschild, cyclic, periodic, and negative homology}, respectively. Consult  Feigin-Tsygan \cite[\S A.6]{FT}\cite[\S4]{FT2} for the computations of these $\cG$-equivariant additive invariants at the small $\cG$-dg categories $\cG \circlearrowright \cC_{c, \dg}(A)$; see Example \ref{ex:twisted-new}(iii).
\end{example}
\begin{example}[Equivariant mixed complex]
The composed functor \eqref{eq:composed-G} with $E:=\mathrm{C}$ is called the {\em $\cG$-equivariant mixed complex}. Let $X$ be a smooth quasi-projective $\cG$-scheme. As proved by Baranovsky in \cite[Thm.~1.1]{Bara}, we have a decomposition
\begin{equation}\label{eq:decomp-Bara}
\mathrm{C}^\cG(\cG \circlearrowright \perf_\dg(X)) \simeq \oplus_{\langle \sigma \rangle} \mathrm{C}(\perf_\dg(X^\sigma))^{C_\cG(\sigma)}\,,
\end{equation}
where $X^\sigma\subset X$ stands for the $\sigma$-invariant subscheme.
\end{example}
\begin{example}[Orbifold cohomology theory]
Since Hochschild, cyclic, periodic, and negative cyclic homology, can be recovered from the mixed complex, the decomposition \eqref{eq:decomp-Bara} holds similarly with $\mathrm{C}$ replaced by $HH, HC, HP$, and $HN$. In the particular case of periodic cyclic homology, with $k=\bbC$, the Hochschild-Kostant-Rosenberg theorem yields a decomposition of $\bbZ/2\bbZ$-graded $\bbC$-vector spaces:
$$ HP^\cG(\cG \circlearrowright \perf_\dg(X))\simeq (\oplus_{\langle \sigma \rangle} H^{\mathrm{even}}(X^\sigma, \bbC)^{C_\cG(\sigma)}, \oplus_{\langle \sigma \rangle} H^{\mathrm{odd}}(X^\sigma, \bbC)^{C_\cG(\sigma)})\,.$$
The right-hand side is known as the even/odd {\em orbifold cohomology} $H^\ast_{\mathrm{orb}}(X,\bbC)$ of the quotient $X/\cG$ in the sense of Chen-Ruan \cite{CR}.
\end{example}
\begin{example}[Equivariant topological Hochschild and cyclic homology]
The composed functor \eqref{eq:composed-G} with $E:=THH$, resp. $E:=TC$, is called {\em $\cG$-equivariant topological Hochschild homology}, resp. {\em $\cG$-equivariant topological cyclic homology}. To the best of the author's knowledge, these invariants are new in the literature.
\end{example}
\begin{proposition}\label{prop:factorization}
Given a $\cG$-equivariant additive invariant $E^\cG$, there exists an additive functor $\overline{E^\cG}\colon \Hmo_0^\cG(k) \to \mathrm{D}$ such that $\overline{E^\cG}\circ U^\cG \simeq E^\cG$.  
\end{proposition}
\begin{proof}
Given two small $\cG$-dg categories $\cG \circlearrowright \cA$ and $\cG \circlearrowright \cB$, consider the dg functor $\rep_\dg(\cA, \cB)^\cG \to \rep_\dg(\cA^\cG,\cB^\cG)$ that sends $(\mathrm{B}\colon \cA\otimes \cB^\op \to \cC_\dg(k), \theta_\sigma)$ to
$$ \cA^\cG \otimes (\cB^\cG)^\op = \cA^\cG \otimes (\cB^\op)^\cG \stackrel{(a)}{\too} (\cA \otimes \cB^\op)^\cG \stackrel{\mathrm{B}^\cG}{\too} \cC_\dg(k)^\cG \stackrel{(b)}{\too} \cC_\dg(k)\,,$$
where (a) stands for the canonical dg functor and (b) for the dg functor which sends a $\cG$-representation $(M,\theta_\sigma)$ to the dg $k$-vector space of $\cG$-invariants $M^\cG$; since $\mathrm{char}(k)\nmid |\cG|$ the latter dg functor is well-defined. By first taking the left dg Kan extension (see \cite[\S4]{Kelly}) of $\rep_\dg(\cA, \cB)^\cG \to \rep_\dg(\cA^\cG,\cB^\cG)$ along the Yoneda dg functor $\rep_\dg(\cA,\cB)^\cG \to \cC_{c, \dg}(\rep_\dg(\cA,\cB)^\cG)$ and then the functor $\dgHo(-)$, we obtain an induced triangulated functor $\rep(\cA,\cB)^\cG \to \rep(\cA^\cG, \cB^\cG)$; see \cite[Thm.~8.7]{Elagin}. Consequently, by passing $K_0$, we obtain an induced homomorphism 
\begin{equation}\label{eq:assignment}
K_0\rep(\cA,\cB)^\cG \too K_0\rep(\cA^\cG, \cB^\cG)\,.
\end{equation} 
The assignments $U^\cG(\cG \circlearrowright \cA) \mapsto U(\cA^\cG)$ and \eqref{eq:assignment} give rise to an additive functor
\begin{eqnarray}\label{eq:induced-additive}
\Hmo_0^\cG(k) \too \Hmo_0(k) && U^\cG(\cG \circlearrowright \cA) \mapsto U(\cA^\cG)
\end{eqnarray}
such that $\eqref{eq:induced-additive}\circ U^\cG \simeq U \circ \eqref{eq:equiv-functor}$. Given a $\cG$-equivariant additive invariant $E^\cG$, let us denote by $\overline{E}\colon \Hmo_0(k) \to \mathrm{D}$ the additive functor corresponding to $E$ under the equivalence of categories \eqref{eq:equivalence-categories}. Under these notations, the additive functor $\overline{E^\cG}$ is now defined as the composition $\overline{E} \circ \eqref{eq:induced-additive}$. 
%
%
%
\end{proof}
\begin{remark}[Green-Julg theorem]
Recall from \eqref{eq:trivial} the definition of the trivial $\cG$-action functor. Clearly, it gives rise to an additive functor
\begin{eqnarray}\label{eq:trivial-additive}
\Hmo_0(k) \too \Hmo_0^\cG(k) && U(\cA) \mapsto U^\cG(\cG \circlearrowright_1 \cA)\,.
\end{eqnarray}
Moreover, similarly to the proof of Proposition \ref{prop:adjunction}, we have the adjunction:
\begin{equation*}
\xymatrix{
\Hmo_0^\cG(k) \ar@<1ex>[d]^-{\eqref{eq:induced-additive}}\\
\Hmo_0(k) \ar@<1ex>[u]^-{\eqref{eq:trivial-additive}}\,.
}
\end{equation*}
By analogy with Kasparov's $KK$-theory (see Meyer's survey \cite{Meyer}), this adjunction may be called the ``Green-Julg theorem''. In particular, given any $\cG$-algebra $\cG \circlearrowright A$, we have a natural isomorphism $K_0(\cD_c(A)^\cG)\simeq K_0(A \rtimes \cG)$.
\end{remark}
\subsection{Enhanced additive invariants}
Given an additive invariant $E$, the associated {\em $\cG$-enhanced additive invariant} is defined as follows
\begin{eqnarray*}
E^\circlearrowright\colon \dgcat^\cG(k) \too \mathrm{D}^\cG && \cG \circlearrowright \cA \mapsto (E(\cA),E(\phi_\sigma))\,,
\end{eqnarray*}
where $\mathrm{D}^\cG$ stands for the category of $\cG$-equivariant objects in $\mathrm{D}$ (with respect to the trivial $\cG$-action); since $E$ sends Morita equivalences to isomorphisms, $E^\circlearrowright$ is well-defined. When $E$ is symmetric monoidal, $E^\circlearrowright$ is also symmetric monoidal. 
\begin{proposition}\label{prop:factorization-enhaced}
Given a $\cG$-enhanced additive invariant $E^\circlearrowright$, there exists an additive functor $\overline{E^\circlearrowright}\colon \Hmo_0^\cG(k) \to \mathrm{D}^\cG$ such that $\overline{E^\circlearrowright}\circ U^\cG \simeq E^\circlearrowright$.
%
\end{proposition}
\begin{proof}
Given small $\cG$-dg categories $\cG \circlearrowright \cA$ and $\cG \circlearrowright\cB$, the composition
\begin{equation}\label{eq:assignment-2}
K_0\rep(\cA,\cB)^\cG \too K_0 \rep(\cA,\cB) \too \Hom_{\mathrm{D}}(E(\cA), E(\cB))\,,
\end{equation}
where the first homomorphism is induced by the restriction functor and the second homomorphism by the additive functor $\overline{E}$, takes values in the abelian subgroup $ \Hom_{\mathrm{D}^\cG}((E(\cA), E(\phi_\sigma)),(E(\cB), E(\phi_\sigma)))$.
Therefore, $\overline{E^\circlearrowright}$ is defined by the assignments $U^\cG(\cG \circlearrowright \cA) \mapsto (E(\cA),E(\phi_\sigma))$ and \eqref{eq:assignment-2}.
\end{proof}
\begin{remark}[Equivariance plus enhacement]
As explained in Remark \ref{rk:character-action}, the functor \eqref{eq:equiv-functor} admits the ``lifting'' \eqref{eq:lifting}. Therefore, given an additive invariant $E$, we can also consider the composition
\begin{equation}\label{eq:comp-new}
\dgcat^\cG(k) \stackrel{\eqref{eq:lifting}}{\too} \dgcat^{\cG^\vee}(k) \stackrel{E^\circlearrowright}{\too} \mathrm{D}^{\cG^\vee}\,.
\end{equation}
Note that by composing \eqref{eq:comp-new} with the forgetful functor from $\mathrm{D}^{\cG^\vee}$ to $\mathrm{D}$, we recover the $\cG$-equivariant additive invariant \eqref{eq:composed-G}. Intuitively speaking, the group $\widehat{\cG}$ of characters acts on every $\cG$-equivariant additive invariant.
\end{remark}
\section{Relation with Panin's motivic category}\label{sec:Panin}
Let $\mathrm{H}$ be an algebraic group scheme over $k$. Recall from Panin \cite[\S6]{Panin}, and from Merkurjev \cite[\S2.3]{Merkurjev}, the construction of the motivic category\footnote{Panin, resp. Merkurjev, denoted this motivic category by $\bbA^{\mathrm{H}}$, resp. $\cC(\mathrm{H})$.} $\cC^{\mathrm{H}}(k)$. The objects are the pairs $(X,A)$, where $X$ is a smooth projective $\mathrm{H}$-scheme and $A$ is a separable algebra, and the morphisms are given by the Grothendieck groups
$$ \Hom_{\cC^\cH(k)}((X,A),(Y,B)):=K_0\mathrm{Vect}^{\mathrm{H}}(X \times Y, A^\op \otimes B)\,,$$
where $\mathrm{Vect}^{\mathrm{H}}(X\times Y, A^\op \otimes B)$ stands for the exact category of those $\mathrm{H}$-equivariant right $(\cO_{X \times Y} \otimes (A^\op \otimes B))$-modules which are locally free and of finite rank as $\cO_{X \times Y}$-modules. Given $[\cF] \in K_0\mathrm{Vect}^{\mathrm{H}}(X\times Y, A^\op \otimes B)$ and $[\mathcal{G}] \in K_0\mathrm{Vect}^{\mathrm{H}}(Y \times Z, B^\op \otimes C)$, their composition is defined by the formula
$$(\pi_{XZ})_\ast (\pi^\ast_{XY}([\cF])\otimes_B \pi^\ast_{YZ}([\mathcal{G}])) \in K_0\mathrm{Vect}^{\mathrm{H}}(X \times Z, A^\op \otimes C)\,,$$
where $\pi_{ST}$ stands for the projection of $X \times Y \times Z$ into $S \times T$. The category $\cC^{\mathrm{H}}(k)$ carries a symmetric monoidal structure induced by $(X,A) \otimes (Y, B) := (X \times Y, A\otimes B)$. Moreover, it comes equipped with two symmetric monoidal functors
\begin{eqnarray}
\SmProj^{\mathrm{H}}(k)^\op \too\cC^{\mathrm{H}}(k) && X\mapsto (X, k)\label{eq:Panin-1} \\
\Sep(k) \too \cC^{\mathrm{H}}(k) && A \mapsto (\mathrm{Spec}(k),A) \label{eq:Panin-2}
\end{eqnarray}
defined on the category of smooth projective $\mathrm{H}$-schemes and separable algebras, respectively. Let us denote by $\dgcat_{\mathrm{sp}}^\cG(k) \subset \dgcat^\cG(k)$ the full subcategory of those small $\cG$-dg categories $\cG \circlearrowright \cA$ with $\cA$ smooth proper.
%
\begin{theorem}\label{thm:bridge-Panin}
When $\mathrm{H}=\cG$ is a (constant) finite algebraic group scheme, there exists an additive, fully faithful, symmetric monoidal functor $\Psi\colon\cC^\cG(k) \to \NChow^\cG(k)$ making the following diagrams commute:
\begin{equation}\label{eq:2diagrams}
\xymatrix{
\SmProj^\cG(k)^\op \ar[d]_-{\eqref{eq:Panin-1}} \ar[rr]^-{X \mapsto \cG \circlearrowright\perf_\dg(X)} && \dgcat_{\mathrm{sp}}^\cG(k) \ar[d]^-{U^\cG} & \mathrm{Sep}(k) \ar[d]_-{\eqref{eq:Panin-2}} \ar[r]^-{A \mapsto \cG \circlearrowright_1 A} & \dgcat_{\mathrm{sp}}^\cG(k) \ar[d]^-{U^\cG} \\
\cC^\cG(k) \ar[rr]_-{\Psi} && \NChow^\cG(k) &
\cC^\cG(k) \ar[r]_-{\Psi}&  \NChow^\cG(k)\,.
}
\end{equation}
\end{theorem}
\begin{proof}
Given a smooth projective $\cG$-scheme $X$ and a separable algebra $A$, let us write $\mathrm{Mod}(X,A)$ for the Grothendieck category of right $(\cO_X \otimes A)$-modules, $\cD(X,A)$ for the derived category $\cD(\mathrm{Mod}(X,A))$, and $\cD_\dg(X,A)$ for the dg category $\cD_\dg(\cE)$ with $\cE:=\mathrm{Mod}(X,A)$. In the same vein, let us write $\perf(X,A)$, resp. $\perf_\dg(X,A)$, for the full triangulated subcategory, resp. full dg subcategory, of those complexes of right $(\cO_X \otimes A)$-modules which are perfect as complexes of $\cO_X$-modules. As proved in \cite[Lem.~6.4]{Homogeneous}, the dg category $\perf_\dg(X,A)$ is smooth proper.

Let $X$ and $Y$ be smooth projective $\cG$-schemes and $A$ and $B$ separable algebras. Consider the inclusion functor 
\begin{equation}\label{eq:inclusion}
\mathrm{Vect}(X\times Y, A^\op \otimes B) \too \perf(X\times Y, A^\op \otimes B)
\end{equation}
as well as the functor
\begin{equation}\label{eq:F-Fourier}
\perf(X\times Y, A^\op \otimes B) \too \rep(\perf_\dg(X,A), \perf_\dg(Y,B)) \quad \cF \mapsto {}_{\Phi_\cF}\mathrm{B} \,,
\end{equation}
where $\Phi_\cF$ stands for the Fourier-Mukai dg functor
\begin{eqnarray*}
\perf_\dg(X,A) \too \perf_\dg(Y,B) && \mathcal{G} \mapsto (\pi_Y)_\ast(\pi_X^\ast(\mathcal{G}) \otimes_A \cF)\,.
\end{eqnarray*}
Both functors \eqref{eq:inclusion}-\eqref{eq:F-Fourier} are $\cG$-equivariant. Consequently, making use of the identification $\perf^\cG(X\times Y, A^\op \otimes B)\simeq \perf(X\times Y, A^\op \otimes B)^\cG$ (see Example \ref{ex:G-equivariant}), we obtain induced group homomorphisms
\begin{equation}\label{eq:induced-1}
K_0\mathrm{Vect}^\cG(X\times Y, A^\op \otimes B) \too K_0\perf^\cG(X\times Y, A^\op \otimes B)
\end{equation} 
\begin{equation}\label{eq:induced-2}
K_0\perf^\cG(X\times Y, A^\op \otimes B) \too K_0\rep(\perf_\dg(X,A), \perf_\dg(Y,B))^\cG\,.
\end{equation}
Similarly to \cite[Thm.~6.10]{Homogeneous}, the assignments $(X,A) \mapsto U^\cG(\cG \circlearrowright \perf_\dg(X,A))$, combined with the group homomorphisms $\eqref{eq:induced-2}\circ \eqref{eq:induced-1}$, gives rise to an additive symmetric monoidal functor $\Psi\colon \cC^\cG(k) \to \NChow^\cG(k)$. As explained at \cite[page~30]{Homogeneous}, the functor \eqref{eq:F-Fourier} is an equivalence. This implies that \eqref{eq:induced-2} is invertible. Since $X\times Y$ admits an ample family of line $\cG$-bundles, the homomorphism \eqref{eq:induced-1} is also invertible. We hence conclude that the functor $\Psi$ is moreover fully faithful. Finally, the commutativity of the diagrams \eqref{eq:2diagrams} follows from the identifications $\perf_\dg(X,k)=\perf_\dg(X)$ and $\perf_\dg(\mathrm{Spec}(k),A)=\cC_{c, \dg}(A)$ and from the fact that the Yoneda dg functor $A \to \cD_{c, \dg}(A)$ is a $\cG$-equivariant Morita equivalence.
\end{proof}
\begin{corollary}
Given $X, Y \in \SmProj^\cG(k)$, we have a group isomorphism
$$ \Hom_{\NChow^\cG(k)}(U^\cG(\cG \circlearrowright \perf_\dg(X)), U^\cG(\cG \circlearrowright \perf_\dg(Y)))\simeq K_0^\cG(X\times Y)\,.$$
\end{corollary}
\begin{proof}
Combine Thomason's definition $K_0^\cG(X\times Y):= K_0\mathrm{Vect}^{\cG}(X\times Y)$ of the $\cG$-equivariant Grothendieck group of $X \times Y$ with Theorem \ref{thm:bridge-Panin}.
\end{proof}
\subsection{Twisted projective homogeneous varieties}
Let $\mathrm{H}$ be a split semi-simple algebraic group scheme over $k$, $P \subset \mathrm{H}$ a parabolic subgroup, and $\gamma\colon \mathrm{Gal}(k_{\mathrm{sep}}/k) \to \cG(k_{\mathrm{sep}})$ a $1$-cocycle. Out of this data, we can construct the projective homogeneous $\mathrm{H}$-variety $\mathrm{H}/P$ as well as its twisted form ${}_\gamma \mathrm{H}/P$. Let $\widetilde{\mathrm{H}}$ and $\widetilde{P}$ be the universal covers of $\mathrm{H}$ and $P$, $R(\widetilde{\mathrm{H}})$ and $R(\widetilde{P})$ the associated representation rings, $n$ the index $[W(\widetilde{\mathrm{H}}):W(\widetilde{P})]$ of the Weyl groups, $\widetilde{Z}$ the center of $\widetilde{\mathrm{H}}$, and $\mathrm{Ch}:=\Hom(\widetilde{Z},\bbG_m)$ the character group. Under these notations, Panin proved in \cite[Thm.~4.2]{Panin} that every $\mathrm{Ch}$-homogeneous basis $\rho_1, \ldots, \rho_n$ of $R(\widetilde{P})$ over $R(\widetilde{\mathrm{H}})$ gives rise to an isomorphism
\begin{equation}\label{eq:Panin}
({}_\gamma \mathrm{H}/P, k)\simeq \oplus_{i=1}^n (\mathrm{Spec}(k), A_i)
\end{equation}
in $\cC^{\mathrm{H}}(k)$, where $A_i$ stands for the Tits' central simple algebra associated to $\rho_i$.
\begin{theorem}\label{thm:Panin}
Let $\mathrm{H}$, $P$, $\gamma$ be as above, and $\cG_k$ the (constant) algebraic group scheme associated to $\cG$. For every homomorphism $\cG_k \to \mathrm{H}$ and $\cG$-equivariant additive invariant $E^\cG$, we have an induced isomorphism 
\begin{equation}\label{eq:Panin1}
E^\cG(\cG \circlearrowright \perf_\dg({}_\gamma \mathrm{H}/P)) \simeq \oplus_{i=1}^n E(A_i[\cG])\,,
\end{equation}
where ${}_\gamma \mathrm{H}/P$ is considered as a $\cG$-scheme.
\end{theorem}
\begin{proof}
Via $\cG_k \to \mathrm{H}$, Panin's computation \eqref{eq:Panin} holds also in the motivic category $\cC^\cG(k)$. Making use of Theorem \ref{thm:bridge-Panin} and Lemma \ref{lem:inverts}, we conclude that 
$$U^\cG(\cG \circlearrowright \perf_\dg({}_\gamma \mathrm{H}/P))\simeq \oplus_{i=1}^n U^\cG(\cG \circlearrowright_1 A_i) \simeq \oplus_{i=1}^n U^\cG(\cG \circlearrowright_1 \cC_{c,\dg}(A_i))\,.$$ 
The proof follows then from Proposition \ref{prop:factorization} and Example \ref{ex:twisted-new}(iii).
\end{proof}
\begin{remark}[$\cG$-equivariant Hochschild homology]
When $E^\cG$ is $\cG$-equivariant Hochschild homology $HH^\cG$, the right-hand side of \eqref{eq:Panin1} reduces to
\begin{equation}\label{eq:isos}
\oplus_{i=1}^n HH(A_i[\cG]) \stackrel{(a)}{\simeq} \oplus_{i=1}^n HH(k[\cG])\otimes HH_0(A_i) \stackrel{(b)}{\simeq} \oplus^n_{i=1} HH(k[\cG])\,,
\end{equation}
where (a) follows from \cite[Cor.~1.2.14]{Loday} and (b) from the fact that $HH_0(A)\simeq k$ for every central simple $k$-algebra $A$. In the particular case where $k$ is algebraically closed, \eqref{eq:isos} reduces moreover to $\oplus^n_{i=1} HH(k)^{\oplus |\langle \cG\rangle|}$; see Remark \ref{rk:additivetwisted}(iii).
\end{remark}
\subsection{Quasi-split case}
When the algebraic group scheme $\mathrm{H}$ is a {\em quasi}-split, Panin proved in \cite[Thm.~12.4]{Panin} that a computation similar to \eqref{eq:Panin} also holds. In this generality, the algebras $A_i$ are no longer central simple but only separable. The analogue of Theorem \ref{thm:Panin} (with the same proof) holds similarly. Moreover, when $E^\cG:=HH^\cG$, the right-hand side of \eqref{eq:Panin1} reduces to $ \oplus^n _{i=1} HH(k[\cG])\otimes A_i/[A_i,A_i]$.
\section{Relation with equivariant motives}\label{sec:relationChow}
\subsection{Equivariant motives}
Given a smooth projective $\cG$-scheme $X$ and an integer $i \in \bbZ$, let us write $\mathrm{CH}_\cG^i(X)_\bbQ$ for the $i^{\mathrm{th}}$-codimensional $\cG$-equivariant Chow group of $X$ in the sense of Edidin-Graham \cite{EG}. Since the group $\cG$ is finite, we have $\mathrm{CH}_\cG^i(X)_\bbQ=0$ whenever $i \notin \{0, \ldots, \mathrm{dim}(X)\}$; see \cite[Prop.~5.2]{Edidin-RR}. 

Let $X$ and $Y$ be smooth projective $\cG$-schemes, $X=\amalg_j X_j$ the decomposition of $X$ into its connected components, and $r$ an integer. The $\bbQ$-vector space $\mathrm{Corr}_\cG^r(X,Y):= \oplus_j \mathrm{CH}_\cG^{\mathrm{dim}(X_j)+r}(X_j \times Y)_\bbQ$ is called the space of {\em $\cG$-equivariant correspondences of degree $r$ from $X$ to $Y$}. Given $\cG$-equivariant correspondences $f \in \mathrm{Corr}_\cG^r(X,Y)$ and $g \in \mathrm{Corr}_\cG^s(Y,Z)$, their composition is defined by the formula
\begin{equation}\label{eq:correspondences}
(\pi_{XZ})_\ast(\pi^\ast_{XY}(f) \cdot \pi^\ast_{YZ}(g)) \in \mathrm{Corr}_\cG^{r+s}(X,Z)\,.
\end{equation}  
Recall from Laterveer \cite{Laterveer}, and from Iyer and M\"uller-Stack \cite{Iyer-Muller}, the construction of the category $\Chow^\cG(k)_\bbQ$ of {\em $\cG$-equivariant Chow motives with $\bbQ$-coefficients}. The objects are the triples $(X,p,m)$, where $X$ is a smooth projective $\cG$-scheme, $p^2=p \in \mathrm{Corr}_\cG^0(X,X)$ is an idempotent endomorphism, and $m$ is an integer. The $\bbQ$-vector spaces of morphisms are given by 
$$ \Hom_{\Chow^\cG(k)_\bbQ}((X,p,m),(Y,q,n)):= q \circ \mathrm{Corr}_\cG^{n-m}(X,Y) \circ p\,,$$
and the composition law is induced by the composition \eqref{eq:correspondences} of correspondences. By construction, the category $\Chow^\cG(k)_\bbQ$ is $\bbQ$-linear, additive, and idempotent complete. Moreover, it carries a symmetric monoidal structure induced by the formula $(X,p,m) \otimes (Y, q,n) := (X \times Y, p \otimes q, m +n)$. The $\cG$-equivariant Lefschetz motive $(\mathrm{Spec}(k), \id, -1)$ will be denoted by $\bbL$ and the $\cG$-equivariant Tate motive $(\mathrm{Spec}(k), \id, 1)$ by $\bbQ(1)$; in both cases $\cG$ acts trivially. Finally, the category $\Chow^\cG(k)_\bbQ$ comes equipped with the symmetric monoidal functor
\begin{eqnarray*}
 \mathfrak{h}^\cG(-)_\bbQ \colon \mathrm{SmProj}^\cG(k)^\op \too \Chow^\cG(k)_\bbQ && X \mapsto (X,\id, 0)\,.
\end{eqnarray*} 
%

The category $\Chow^\cG(k)_\bbQ$ is additive and rigid symmetric monoidal. Therefore, similarly to \S\ref{sub:NNmotives}, the category $\Num^\cG(k)_\bbQ$ of {\em $\cG$-equivariant numerical motives with $\bbQ$-coefficients} is defined as the idempotent completion of $\Chow^\cG(k)_\bbQ/ \cN$.
\subsection{Orbit categories}\label{sub:orbit}
Let $\cC$ be an additive symmetric monoidal category and $\cO\in \cC$ a $\otimes$-invertible object. The {\em orbit category} $\cC/_{\!-\otimes \cO}$ has the same objects as $\cC$ and abelian groups of morphisms  $\Hom_{\cC/_{\!-\otimes \cO}}(a,b):=\oplus_{i\in \bbZ} \Hom_\cC(a,b \otimes \cO^{\otimes i})$. Given objects $a$, $b$, and $c$, and morphisms
\begin{eqnarray*}
\mathrm{f}=\{f_i\}_{i \in \bbZ} \in \oplus_{i \in \bbZ} \Hom_\cC(a, b \otimes \cO^{\otimes i}) && \mathrm{g}=\{g_i\}_{i \in \bbZ} \in \oplus_{i \in \bbZ} \Hom_\cC(b, c \otimes \cO^{\otimes i})
\end{eqnarray*} 
the $i'^{\mathrm{th}}$-component of $\mathrm{g} \circ \mathrm{f}$ is defined as $\sum_i (g_{i'-i}\otimes \cO^{\otimes i})\circ f_i$. The functor
\begin{eqnarray*}
\pi\colon \cC \too \cC/_{\!-\otimes \cO} & a \mapsto a & f \mapsto \mathrm{f}=\{f_i\}_{i \in \bbZ}\,,
\end{eqnarray*}
where $f_0=f$ and $f_i=0$ if $i \neq 0$, is endowed with a natural isomorphism of functors $\pi \circ (-\otimes \cO) \Rightarrow \pi$ and is $2$-universal among all such functors; see \cite[\S7]{CvsNC}. The category $\cC/_{\!-\otimes \cO}$ is additive and, as proved in \cite[Lem.~7.3]{CvsNC}, it inherits from $\cC$ a symmetric monoidal structure making $\pi$ symmetric monoidal.
\subsection{Localization at the augmentation ideal}\label{sub:augmentation}
Let $I$ be the kernel of the rank homomorphism $R(\cG) \twoheadrightarrow \bbZ$ and $R(\cG)_{I}$ the localization of $R(\cG)$ at the ideal $I$. Recall from Corollary \ref{cor:R(G)-linear} that the category $\Hmo_0^\cG(k)$ is $R(\cG)$-linear. Let us denote by $\Hmo_0^\cG(k)_I$ the $R(\cG)_I$-linear additive category obtained from $\Hmo_0^\cG(k)$ by applying the functor $(-)_I:= - \otimes_{R(\cG)}R(\cG)_I$ to each $R(\cG)$-module of morphisms. By construction, $\Hmo_0^\cG(k)_I$ inherits from $\Hmo^\cG_0(k)$ a symmetric monoidal structure making the functor $(-)_I\colon \Hmo_0^\cG(k) \to \Hmo_0^\cG(k)_I$ symmetric monoidal. The category $\NChow^\cG(k)_I$ of {\em $I$-localized $\cG$-equivariant noncommutative Chow motives} is defined as the idempotent completion of the subcategory of $\Hmo_0^\cG(k)_I$ consists of the objects $U^\cG(\cG \circlearrowright \cA)_I$ with $\cA$ a smooth proper dg category. The category $\NNum^\cG(k)_I$ of {\em $I$-localized $\cG$-equivariant noncommutative numerical motives} is defined as the idempotent completion of the category obtained from $\NNum^\cG(k)$ by applying the functor $(-)_I$ to each $R(\cG)$-module of morphisms; see \S\ref{sub:NCnumerical}.
\begin{proposition}\label{prop:same}
Given any two cohomology classes $[\alpha], [\beta] \in H^2(\cG, k^\times)$, we have an isomorphism $U^\cG(\cG \circlearrowright_\alpha k)_I\simeq U^\cG(\cG \circlearrowright_\beta k)_I$ in $\NChow^\cG(k)_I$.
\end{proposition}
\begin{proof}
By construction of $\NChow^\cG(k)$, we have group isomorphisms:
$$ \Hom_{\NChow^\cG(k)}(U^\cG(\cG \circlearrowright_\alpha k), U^\cG(\cG \circlearrowright_\beta k))\simeq K_0 (\cD_c(k)^{\cG, \alpha\beta^{-1}})$$
$$ \Hom_{\NChow^\cG(k)}(U^\cG(\cG \circlearrowright_\beta k), U^\cG(\cG \circlearrowright_\alpha k))\simeq K_0 (\cD_c(k)^{\cG, \beta\alpha^{-1}})\,.$$
Consider the $\alpha\beta^{-1}$-twisted $\cG$-equivariant object $k_{\alpha\beta^{-1}} \cG \in \cD_c(k)^{\cG, \alpha\beta^{-1}}$ defined as $(\oplus_{\rho \in \cG} \phi_\rho(k), \theta_\sigma)$, where $\phi_\rho(k)=k$ and $\theta_\sigma$ is given by the collection of units $(\alpha^{-1} \beta)(\sigma, \rho)\in k^\times$. Similarly, consider the $\beta\alpha^{-1}$-twisted $\cG$-equivariant object $k_{\beta\alpha^{-1}} \cG \in \cD_c(k)^{\cG, \beta\alpha^{-1}}$ defined as $(\oplus_{\rho \in \cG} \phi_\rho(k), \theta_\sigma)$, where $\theta_\sigma$ is given by $(\beta^{-1} \alpha)(\sigma, \rho)$. The associated Grothendieck classes correspond then to morphisms
\begin{eqnarray*}
U^\cG(\cG \circlearrowright_\alpha k) \stackrel{f}{\too} U^\cG(\cG \circlearrowright_\beta k) & \mathrm{and} &  U^\cG(\cG \circlearrowright_\beta k) \stackrel{g}{\too} U^\cG(\cG \circlearrowright_\alpha k) 
\end{eqnarray*}
in the category $\NChow^\cG(k)$. Since the rank of the elements $g \circ f, f \circ g \in R(\cG)$ is non-zero (see Proposition \ref{prop:endomorphisms}), we conclude from the definition of $\NChow^\cG(k)_I$ that the morphisms $f_I$ and $g_I$ are invertible. This achieves the proof.
\end{proof}
\begin{remark}[Groups of central type]
Note that the group algebra $k[\cG]$ is not simple; it contains the non-trivial augmentation ideal. In the case where $\cG$ is of central type, there exist cohomology classes $[\alpha] \in H^2(\cG,k^\times)$ for which the twisted group algebra $k_\alpha[\cG]$ is simple! For example, the group $\cG :=\mathrm{H} \times \widehat{\mathrm{H}}$ (with $\mathrm{H}$ abelian) is of central type and the twisted group algebra $k_\alpha[\cG]$ associated to the $2$-cocycle $\alpha((\sigma,\chi),(\rho,\psi)):=\chi(\rho)$ is simple. By combining Remark \ref{rk:additivetwisted} with Example \ref{ex:twisted-new}(iv) and with Proposition \ref{prop:factorization}, we conclude that $U^\cG(\cG \circlearrowright_1 k) \not\simeq U^\cG(\cG \circlearrowright_\alpha k)$ in $\NChow^\cG(k)$. This shows that Proposition \ref{prop:same} is false before $I$-localization.
\end{remark}
\subsection{Bridges}
The next result relates the categories of $\cG$-equivariant noncommutative motives with the categories of $\cG$-equivariant motives.
\begin{theorem}\label{thm:bridge}
There exist $\bbQ$-linear, fully-faithful, symmetric monoidal functors $\Phi$ and $\Phi_\cN$ making the following diagram commute:
\begin{equation}\label{eq:diagram-bridge}
\xymatrix@C=3em@R=2em{
\mathrm{SmProj}^\cG(k)^\op \ar[rrr]^-{X \mapsto \cG \circlearrowright \perf_\dg(X)} \ar[d]_-{\mathfrak{h}^\cG(-)_\bbQ} &&& \dgcat_{\mathrm{sp}}^\cG(k) \ar[d]^-{U^\cG(-)_\bbQ} \\
\Chow^\cG(k)_\bbQ \ar[d]_-\pi &&& \NChow^\cG(k)_\bbQ \ar[d]^-{(-)_{I_\bbQ}} \\
\Chow^\cG(k)_\bbQ/_{\!-\otimes \bbQ(1)} \ar[rrr]_-{\Phi} \ar[d] &&& \NChow^\cG(k)_{\bbQ, I_\bbQ} \ar[d] \\
\Num^\cG(k)_\bbQ/_{\!-\otimes \bbQ(1)} \ar[rrr]_{\Phi_\cN} &&& \NNum^\cG(k)_{\bbQ,I_\bbQ}\,.
}
\end{equation}
\end{theorem}
\begin{proof}
Let us denote by $\cC_{\mathrm{sp}}^\cG(k)_\bbQ$ the idempotent completion of the full subcategory of $\cC^\cG(k)_\bbQ$ (see \S\ref{sec:Panin}) consisting of the objects $(X,k)_\bbQ$. Given smooth projective $\cG$-schemes $X$ and $Y$, we have isomorphisms
$$\Hom_{\cC_{\mathrm{sp}}^\cG(k)_\bbQ}(\mathfrak{h}_0^\cG(X)_\bbQ, \mathfrak{h}_0^\cG(Y)_\bbQ) = K_0 \mathrm{Vect}^\cG(X\times Y)_\bbQ \simeq K_0^\cG(X\times Y)_\bbQ\,.$$
 Moreover, given $[\cF]_\bbQ \in K_0^\cG(X\times Y)_\bbQ$ and $[\mathcal{G}]_\bbQ \in K_0^\cG(Y\times Z)_\bbQ$, their composition is defined by the formula $(\pi_{XZ})_\ast(\pi^\ast_{XY}([\cF]_\bbQ)\otimes \pi^\ast_{YZ}([\mathcal{G}]_\bbQ))$. Furthermore, $\cC^\cG_{\mathrm{sp}}(k)_\bbQ$ comes equipped with the symmetric monoidal functor 
\begin{eqnarray*}
\mathfrak{h}^\cG_0(-)\colon \mathrm{SmProj}^\cG(k)^\op \too \cC_{\mathrm{sp}}^\cG(k)_\bbQ && X\mapsto (X,k)_\bbQ\,.
\end{eqnarray*}
Similarly to \S\ref{sub:augmentation}, we can also consider the $I_\bbQ$-localized category $\cC_{\mathrm{sp}}^\cG(k)_{\bbQ, I_\bbQ}$. 

Let us now construct a functor $\Phi_1$ making the following diagram commute
\begin{equation}\label{eq:diagram-aux2}
\xymatrix@C=3em@R=2em{
\mathrm{SmProj}^\cG(k)^\op \ar@{=}[r] \ar[d]_-{\mathfrak{h}^\cG(-)_\bbQ} & \mathrm{SmProj}(k)^\op \ar[rr]^-{X \mapsto \cG \circlearrowright \perf_\dg(X)} \ar[d]^-{\mathfrak{h}^\cG_0(-)_\bbQ} && \dgcat_{\mathrm{sp}}^\cG(k) \ar[d]^-{U^\cG(-)_\bbQ} \\
\Chow^\cG(k)_\bbQ \ar[d]_-\pi &\cC_{\mathrm{sp}}^\cG(k)_\bbQ \ar[d]^-{(-)_{I_\bbQ}}&& \NChow^\cG(k)_\bbQ \ar[d]^-{(-)_{I_\bbQ}} \\
\Chow^\cG(k)_\bbQ/_{\!-\otimes \bbQ(1)} &\cC_{\mathrm{sp}}^\cG(k)_{\bbQ, I_\bbQ} \ar[l]^-{\Phi_1} \ar[rr]_-{\Phi_2}&& \NChow^\cG(k)_{\bbQ, I_\bbQ} \,,
}
\end{equation}
where $\Phi_2$ stands for the $\bbQ$-linear, fully faithful, symmetric monoidal functor naturally induced from $\Psi$; see Theorem \ref{thm:bridge-Panin}. As proved in \cite[Cor.~5.1]{Edidin-RR}, we have a Riemann-Roch isomorphism $ \tau_X\colon K_0^\cG(X)_{\bbQ,I_\bbQ} \to \oplus_{i=0}^{\mathrm{dim}(X)}\mathrm{CH}^i_\cG(X)_\bbQ$ for every smooth projective $\cG$-scheme $X$. This isomorphism preserves the multiplicative structures. Moreover, given any $\cG$-equivariant map $f\colon X \to Y$, the following squares are commutative (we assume that $f$ is proper on the right-hand side):
$$
\xymatrix{
K_0^\cG(X)_{\bbQ, I_\bbQ}\ar[r]^-{\tau_X} & \oplus_{i=0}^{\mathrm{dim}(X)} \mathrm{CH}_\cG^i(X)_\bbQ  & K_0^\cG(X)_{\bbQ, I_\bbQ}\ar[r]^-{\tau_X} \ar[d]_-{f_\ast}& \oplus_{i=0}^{\mathrm{dim}(X)} \mathrm{CH}_\cG^i(X)_\bbQ \ar[d]^-{f_\ast} \\
 K_0^\cG(Y)_{\bbQ, I_\bbQ} \ar[u]^-{f^\ast} \ar[r]_-{\tau_Y} & \oplus_{i=0}^{\mathrm{dim}(Y)} \mathrm{CH}_\cG^i(Y)_\bbQ  \ar[u]_-{f^\ast} & K_0^\cG(Y)_{\bbQ, I_\bbQ}\ar[r]_-{\tau_Y} & \oplus_{i=0}^{\mathrm{dim}(Y)} \mathrm{CH}_\cG^i(Y)_\bbQ\,.
}
$$
By construction of the orbit category, we have isomorphisms
$$  \Hom_{\Chow^\cG(k)_\bbQ/_{\!-\otimes \bbQ(1)}}(\pi(\mathfrak{h}^\cG(X)_\bbQ),\pi(\mathfrak{h}^\cG(Y)_\bbQ))\simeq \oplus_{i=0}^{\mathrm{dim}(X\times Y)} \mathrm{CH}^i_\cG(X\times Y)_\bbQ\,.$$
Therefore, we conclude from the preceding considerations that the assignments
\begin{eqnarray*}
\mathfrak{h}_0^\cG(X)_\bbQ \mapsto \mathfrak{h}^\cG(X)_\bbQ &\mathrm{and}& K_0^\cG(X\times Y)_{\bbQ, I_\bbQ} \stackrel{\tau_{X\times Y}}{\too} \oplus_{i=0}^{\mathrm{dim}(X\times Y)} \mathrm{CH}^i_\cG(X\times Y)_\bbQ
\end{eqnarray*}
give rise to a functor $\Phi_1\colon \cC_{\mathrm{sp}}^\cG(k)_{\bbQ, I_\bbQ} \to \Chow^\cG(k)_\bbQ/_{\!-\otimes \bbQ(1)}$ making the diagram \eqref{eq:diagram-aux2} commute. The functor $\Phi_1$ is $\bbQ$-linear, fully-faithful, and symmetric monoidal. Since the objects $(X,p,m)$ and $(X,p, 0)$ become isomorphic in the orbit category $\Chow^\cG(k)_\bbQ/_{\!-\otimes \bbQ(1)}$, the functor $\Phi_1$ is moreover essentially surjective and consequently an equivalence of categories. Now, choose a (quasi-)inverse functor $\Phi_1^{-1}$ of $\Phi_1$ and define $\Phi$ as the composition $\Phi_2 \circ \Phi_1^{-1}$. By construction, $\Phi$ is $\bbQ$-linear, fully faithful, symmetric monoidal, and makes the upper rectangle of \eqref{eq:diagram-bridge} commute. 

Now, consider the following commutative diagram:
$$
\xymatrix@C=1.8em@R=2em{
\Chow^\cG(k)_\bbQ/_{\!-\otimes \bbQ(1)} \ar@{=}[r] \ar[d]& \Chow^\cG(k)_\bbQ/_{\!-\otimes \bbQ(1)} \ar[d] \ar[r]^-{\Phi} & \NChow^\cG(k)_{\bbQ, I_\bbQ} \ar[d] \\
\Num^\cG(k)_\bbQ/_{\!-\otimes \bbQ(1)} & \ar[l]^-{\Phi'_\cN} ((\Chow^\cG(k)_\bbQ/_{\!-\otimes \bbQ(1)})/\cN)^\natural \ar[r]_-{\Phi''_\cN} & ((\NChow^\cG(k)_{\bbQ, I_\bbQ})/\cN)^\natural\,,
}
$$
where $(-)^\natural$ stands for the idempotent completion construction. The functor $\Phi'_\cN$, whose construction follows from the general result \cite[Prop.~3.2]{AJM}, is an equivalence of categories. In what concerns $\Phi''_\cN$, it is naturally induced from $\Phi$. In the construction of $\NNum^\cG(k)_\bbQ$, we can consider $\NChow^\cG(k)_\bbQ$ as a $\bbQ$-linear category or as a $R(\cG)_\bbQ$-linear category. Making use of \cite[Prop.~1.4.1]{Brugieres}, we conclude that $((\NChow^\cG(k)_{\bbQ, I_\bbQ})/\cN)^\natural$ is naturally equivalent to the category $\NNum^\cG(k)_{\bbQ, I_\bbQ}$. Now, choose a (quasi-)inverse $(\Phi'_\cN)^{-1}$ of $\Phi'_\cN$ and define $\Phi_\cN$ as the composition $\Phi''_\cN \circ (\Phi'_\cN)^{-1}$. By construction, the functor $\Phi_\cN$ is $\bbQ$-linear, fully faithful, symmetric monoidal, and makes the bottom of diagram \eqref{eq:diagram-bridge} commute.
\end{proof}
\section{Full exceptional collections}\label{sec:decompositions}
\subsection{Full exceptional collections}
Let $\cT$ be a $k$-linear triangulated category. Recall from Bondal-Orlov \cite[Def.~2.4]{BO} and Huybrechts \cite[\S1.4]{Huybrechts} that a {\em semi-orthogonal decomposition of length $n$}, denoted by $\cT=\langle \cT_1, \ldots, \cT_n\rangle$, consists of full triangulated subcategory $\cT_1, \ldots, \cT_n \subset \cT$ satisfying the following conditions: the inclusions $\cT_i \subset \cT$ admit left and right adjoints; the triangulated category $\cT$ is generated by the objects of $\cT_1, \ldots, \cT_n$; and $\Hom_\cT(\cT_j, \cT_i)=0$ when $i <j$. An object $\cE \in \cT$ is called {\em exceptional} if $\Hom_\cT(\cE,\cE)=k$ and $\Hom_\cT(\cE, \cE[m])=0$ when $m \neq 0$. A {\em full exceptional collection of length $n$}, denoted by $\cT=(\cE_1, \ldots, \cE_n)$, is a sequence of exceptional objects $\cE_1, \ldots, \cE_n$ which generate the triangulated category $\cT$ and for which we have $\Hom_\cT(\cE_j, \cE_i[m])=0, m \in \bbZ$, when $i <j$. Every full exceptional collection gives rise to a semi-orthogonal decomposition $\cT=\langle \cD_c(k), \ldots, \cD_c(k)\rangle$.


\begin{proposition}\label{prop:semiorthogonal}
Let $\cA$ be a small $\cG$-dg category and $\cA_i \subseteq \cA, 1 \leq i \leq n$, full dg subcategories. Assume that $\sigma^\ast(\cA_i)\subseteq \cA_i$ for every $\sigma \in \cG$, and that $\cD_c(\cA)$ admits a semi-orthogonal decomposition $\langle \cD_c(\cA_1), \ldots, \cD_c(\cA_n)\rangle$. Under these assumptions, we have an isomorphism $U^\cG(\cG \circlearrowright \cA) \simeq \oplus_{i=1}^n U^\cG(\cG \circlearrowright \cA_i)$ in $\Hmo_0^\cG(k)$.
\end{proposition}
\begin{proof}
The inclusions of dg categories $\cA_i \subseteq \cA$ give rise to a morphism
\begin{equation}\label{eq:induced-semi}
\oplus_{i=1}^n U^\cG(\cG \circlearrowright \cA_i) \too U^\cG(\cG \circlearrowright \cA)
\end{equation}
in the category $\Hmo^\cG_0(k)$. In order to show that \eqref{eq:induced-semi} is an isomorphism, it suffices by the Yoneda lemma to show that the induced group homomorphism
$$\Hom(U^\cG(\cG \circlearrowright \cB),\oplus_{i=1}^n U^\cG(\cG \circlearrowright \cA_i)) \too \Hom(U^\cG(\cG \circlearrowright \cB), U^\cG(\cG \circlearrowright \cA))$$
is invertible for every small $\cG$-dg category $\cG \circlearrowright \cB$. By construction of the additive category $\Hmo^\cG_0(k)$, the preceding homomorphism identifies with 
\begin{equation}\label{eq:associated-1}
\oplus^n_{i=1} K_0\rep(\cB,\cA_i)^\cG \too K_0\rep(\cB, \cA)^\cG\,.
\end{equation}
Since $\cD_c(\cA)=\langle\cD_c(\cA_1), \ldots, \cD_c(\cA_n)\rangle$, we have a semi-orthogonal decomposition
$$ \rep(\cB,\cA) = \langle \rep(\cB,\cA_1), \ldots, \rep(\cB,\cA_n)\rangle\,.$$
Using first the fact that the functor $(-)^\cG$ preserves semi-orthogonal decompositions, and then the fact that the functor $K_0(-)$ sends semi-orthogonal decompositions to direct sums, we conclude that the group homomorphism \eqref{eq:associated-1} is invertible.
\end{proof}
\subsection{Invariant objects}\label{sub:invariants}
Let $\cG \circlearrowright \cA$ be a small $\cG$-dg category. An object $M \in \cD(\cA)$ is called {\em $\cG$-invariant} if $\phi_\sigma(M)\simeq M$ for every $\sigma \in \cG$. Every $\cG$-equivariant object in $\cG \circlearrowright \cD(\cA)$ is $\cG$-invariant, but the converse does not hold.
\begin{remark}[Strictification]\label{rk:2-cocycle}
Given a $\cG$-invariant object $M \in \cD(\cA)$, let us fix an isomorphism $\theta_\sigma \colon M \to \phi_\sigma(M)$ for every $\sigma \in \cG$. If $\Hom_{\cD(\cA)}(M,M)\simeq k$, then $\phi_\rho(\theta_\sigma) \circ \theta_\rho$ and $\theta_{\rho \sigma}$
differ by multiplication with an invertible element $\alpha(\rho,\sigma) \in k^\times$. Moreover, these invertible elements define a $2$-cocycle $\alpha$ whose cohomology class $[\alpha] \in H^2(\cG,k^\times)$ is independent of the choice of the $\theta_\sigma$'s. As a consequence, $M \in \cD(\cA)^{\cG, \alpha}$. Furthermore, $M^{\otimes n} \in \cD(\cA)^{\cG, \alpha^n}$. Roughly speaking, every ``simple'' $\cG$-invariant object can be strictified into~a~twisted~$\cG$-equivariant~object. 
\end{remark}
\begin{proposition}\label{prop:exceptional}
Let $\cA$ be a small $\cG$-dg category such that $\cD_c(\cA)$ admits a full exceptional collection $(\cE_1, \ldots, \cE_n)$. Suppose that $\cE_i\in \cD_c(\cA)^{\cG, \alpha_i}$, with $[\alpha_i] \in H^2(\cG,k^\times)$. Then, we have $U^\cG(\cG \circlearrowright \cA)\simeq \oplus_{i=1}^n U^\cG(\cG \circlearrowright_{\alpha_i} k)$ in $\Hmo_0^\cG(k)$.
\end{proposition}

\begin{proof}
By construction, the set of morphisms $\Hom_{\Hmo^\cG(k)}(\cG \circlearrowright_{\alpha_i} k, \cG \circlearrowright \cA)$ is given by the set of isomorphism classes of the triangulated category $\rep(k, \cA)^{\cG, \alpha_i}\simeq \cD_c(\cA)^{\cG, \alpha_i}$. Consequently, the object $\cE_i \in  \cD_c(\cA)^{\cG, \alpha_i}$ corresponds to a morphism $\cE_i \colon \cG \circlearrowright_{\alpha_i} k \to \cG \circlearrowright \cA$ in $\Hmo^\cG(k)$. Consider the associated morphism 
\begin{equation}\label{eq:associated}
([\cE_1]\ldots [\cE_i]\ldots [\cE_n])\colon \oplus_{i=1}^n U^\cG(\cG \circlearrowright_{\alpha_i} k) \too U^\cG(\cG \circlearrowright \cA)
\end{equation}
in the additive category $\Hmo^\cG_0(k)$. In order to show that \eqref{eq:associated} is an isomorphism, we can now follow {\em mutatis mutandis} the proof of Proposition \ref{prop:semiorthogonal}.
%
\end{proof}
\begin{corollary}
Given a $\cG$-dg category $\cG \circlearrowright \cA$ as in Proposition \ref{prop:exceptional}, we have:
\begin{itemize}
\item[(i)] $E^\cG(\cG \circlearrowright \cA) \simeq \oplus_{i=1}^n E(k_{\alpha_i}[\cG])$ for every $\cG$-equivariant additive invariant;
\item[(ii)] $E^\circlearrowright(\cG \circlearrowright \cA) \simeq \oplus_{i=1}^n (E(k),\id)$ for every $\cG$-enhanced additive invariant.
\end{itemize}
\end{corollary}
\begin{proof}
Item (i) follows from the combination of Propositions \ref{prop:factorization} and \ref{prop:exceptional} with Example \ref{ex:twisted-new}(iv). Item (ii) follows from the combination of Propositions \ref{prop:factorization-enhaced} and \ref{prop:exceptional} with the fact that $E^\circlearrowright(\cG \circlearrowright_\alpha k) \simeq (E(k),\id)$ for every $[\alpha] \in H^2(\cG,k^\times)$.
\end{proof}
\begin{proposition}\label{prop:exceptional-schemes}
Let $X$ be a quasi-compact quasi-separated $\cG$-scheme such that $\perf(X)$ admits a full exceptional collection $(\cE_1, \ldots, \cE_n)$ of $\cG$-invariant objects. Let us denote by $[\alpha_i] \in H^2(\cG, k^\times)$ the cohomology class of Remark \ref{rk:2-cocycle} associated to the exceptional object $\cE_i$. Under these assumptions and notations, we have an isomorphism $U^\cG(\cG \circlearrowright \perf_\dg(X)) \simeq \oplus_{i=1}^n U^\cG(\cG \circlearrowright_{\alpha_i} k)$ in $\Hmo^\cG_0(k)$.
\end{proposition}
\begin{proof}
Apply Proposition \ref{prop:exceptional} to the dg category $\perf_\dg(X)$. 
\end{proof}
\begin{example}[Projective spaces]
Let $\mathbb{P}^n$ be the $n^{\mathrm{th}}$ projective space. As proved by Beilinson in \cite{Beilinson}, $\perf(\mathbb{P}^n)$ admits a full exceptional collection  $(\cO, \cO(1), \ldots, \cO(n))$. Moreover, the objects $\cO(i)$ are $\cG$-invariant for any $\cG$-action on $\bbP^n$. Let us denote by $[\alpha]$ the cohomology class of Remark \ref{rk:2-cocycle} associated to the exceptional object $\cO(1)$. Under these notations, Proposition \ref{prop:exceptional-schemes} yields an isomorphism
$$ U^\cG(\cG \circlearrowright \perf_\dg(\bbP^n)) \simeq U^\cG(\cG \circlearrowright_1 k)\oplus U^\cG(\cG \circlearrowright_\alpha k)\oplus  \cdots \oplus U^\cG(\cG \circlearrowright_{\alpha^n} k)\,.$$
\end{example}
\begin{example}[Odd dimensional quadrics]\label{ex:quadrics}
Assume that $\mathrm{char}(k)\neq 2$. Let $(V,q)$ be a non-degenerate quadratic form of odd dimension $n\geq 3$ and $Q_q\subset \bbP(V)$ the associated smooth projective quadric of dimension $d:=n-2$. As proved by Kapranov in \cite{Kapranov}, $\perf(Q_q)$ admits a full exceptional collection $(\cS, \cO, \cO(1), \cdots, \cO(d-1))$, where $\cS$ denotes the spinor bundle. Moreover, the objects $\cO(i)$ and $\cS$ are $\cG$-invariant for any $\cG$-action on $Q_q$; see \cite[\S3.2]{Elagin-semiorthogonal}. Let us denote by $[\alpha]$ and $[\beta]$ the cohomology classes of Remark \ref{rk:2-cocycle} associated to the exceptional object $\cO(1)$ and $\cS$, respectively. Under these notations, Proposition \ref{prop:exceptional-schemes} yields an isomorphism between $U^\cG(\cG\circlearrowright \perf_\dg(Q_q))$ and the direct sum
$$U^\cG(\cG \circlearrowright_\beta k)\oplus U^\cG(\cG \circlearrowright_1k)\oplus U^\cG(\cG \circlearrowright_\alpha k)  \oplus \cdots \oplus U^\cG(\cG \circlearrowright_{\alpha^{(d-1)}}k)\,.$$
\end{example}
\begin{example}[Grassmannians]
Assume that $\mathrm{char}(k)=0$. Let $V$ be a $k$-vector space of dimension $d$, $n\leq d$ a positive integer, and $\mathrm{Gr}:=\mathrm{Gr}(n,V)$ the Grassmannian of $n$-dimensional subspaces in $V$. As proved by Kapranov in \cite{Kapranov}, $\perf(\mathrm{Gr})$ admits a full exceptional collection $(\cO, \cU^\vee, \ldots, \Sigma^{\lambda}_{n(d-n)}\cU^\vee)$, where $\cU^\vee$ denotes the dual of the tautological vector bundle on $\mathrm{Gr}$ and $\Sigma^\lambda_i$ the Schur functor associated to a Young diagram $\lambda$ with $|\lambda|=i$ having at most $n$ rows and $d-n$ rows. Moreover, the objects $\Sigma^\lambda_i\cU^\vee$ are $\cG$-invariant for any $\cG$-action on $Q_q$ which is induced by an homomorphism $\cG \to \mathrm{PGL}(V)$. Let us denote by $[\alpha]$ the cohomology class of Remark \ref{rk:2-cocycle} associated to the exceptional object $\cU^\vee$. Under these notations, Proposition \ref{prop:exceptional-schemes} yields an isomorphism
$$ U^\cG(\cG\circlearrowright \perf_\dg(\mathrm{Gr})) \simeq U^\cG(\cG \circlearrowright_1 k) \oplus U^\cG(\cG \circlearrowright_\alpha k) \oplus \cdots \oplus (\oplus_{\lambda} U^\cG(\cG \circlearrowright_{\alpha^{n(d-n)}}k))\,.
$$
%
\end{example}
\subsection*{Proof of Theorem \ref{thm:via}}
In order to simplify the exposition, let us write $\mathfrak{h}^\cG(X)_\bbQ(i)$ instead of $\mathfrak{h}^\cG(X)_\bbQ\otimes \bbQ(1)^{\otimes i}$. Following Remark \ref{rk:2-cocycle}, let us denote by $[\alpha_i] \in H^2(\cG, k^\times)$ the cohomology class associated to the exceptional object $\cE_i$. By combining Propositions \ref{prop:same} and \ref{prop:exceptional-schemes}, we obtain induced isomorphisms
$$U^\cG(\cG \circlearrowright \perf_\dg(X))_{\bbQ, I_\bbQ} \simeq \oplus_{i=1}^n U^\cG(\cG \circlearrowright_{\alpha_i} k)_{\bbQ, I_\bbQ}\simeq  \oplus_{i=1}^n U^\cG(\cG \circlearrowright_1 k)_{\bbQ, I_\bbQ}$$ in the category $\Hmo_0^\cG(k)_{\bbQ, I_\bbQ}$. Since $\mathfrak{h}^\cG(\mathrm{Spec}(k))_\bbQ$ (with trivial $\cG$-action) is the $\otimes$-unit of $\Chow^\cG(k)_\bbQ$ and $U^\cG(\cG \circlearrowright_1 k)_{\bbQ, I_\bbQ}$ the $\otimes$-unit of $\NChow^\cG(k)_{\bbQ, I_\bbQ}$, we conclude from Theorem \ref{thm:bridge} that $\pi(\mathfrak{h}^\cG(X)_\bbQ)$ is isomorphic to $\oplus_{j=1}^n\pi(\mathfrak{h}^\cG(\mathrm{Spec}(k))_\bbQ)$ in the orbit category $\Chow^\cG(k)_\bbQ/_{\!-\otimes \bbQ(1)}$. Let us now ``lift'' this isomorphism to the category $\Chow^\cG(k)_\bbQ$. Since the functor $\pi$ is additive, there exist morphisms
$$ \mathrm{f} =\{f_i\}_{i \in \bbZ} \in \oplus_{i \in \bbZ} \Hom_{\Chow^\cG(k)_\bbQ}(\mathfrak{h}^\cG(X)_\bbQ, \oplus_{j=1}^n \mathfrak{h}^\cG(\mathrm{Spec}(k))_\bbQ(i))$$
$$ \mathrm{g} =\{g_i\}_{i \in \bbZ} \in \oplus_{i \in \bbZ} \Hom_{\Chow^\cG(k)_\bbQ}(\oplus_{j=1}^n \mathfrak{h}^\cG(\mathrm{Spec}(k))_\bbQ, \mathfrak{h}^\cG(X)_\bbQ(i))$$
verifying the equalities $\mathrm{g}\circ \mathrm{f} = \id = \mathrm{f} \circ \mathrm{g}$. Moreover, as explained in \S\ref{sec:relationChow}, we have
$$ \Hom_{\Chow^\cG(k)_\bbQ}(\mathfrak{h}^\cG(X)_\bbQ, \oplus^n_{j=1} \mathfrak{h}^\cG(\mathrm{Spec}(k))_\bbQ(i))\simeq \oplus_{j=1}^n \mathrm{CH}_\cG^{\mathrm{dim}(X)+i}(X)_\bbQ$$
$$\Hom_{\Chow^\cG(k)_\bbQ}(\oplus^n_{j=1} \mathfrak{h}^\cG(\mathrm{Spec}(k))_\bbQ, \mathfrak{h}^\cG(X)_\bbQ(i)) \simeq \oplus^n_{j=1} \mathrm{CH}^i_\cG(X)_\bbQ\,.$$
This implies that $f_i=0$ when $i \notin \{ -\mathrm{dim}(X), \ldots, 0\}$ and that $g_i=0$ when $i \notin \{0, \ldots, \mathrm{dim}(X)\}$. The sets $\{ f_{-r}\,|\, 0 \leq r \leq \mathrm{dim}(X)\}$ and $\{g_r(-r)\,|\, 0 \leq r\leq \mathrm{dim}(X)\}$ give then rise to morphisms in the category of $\cG$-equivariant Chow motives:
\begin{equation}\label{eq:morph-comp-1}
\mathfrak{h}^\cG(X)_\bbQ \too \oplus^{\mathrm{dim}(X)}_{r=0} \oplus_{j=1}^n \mathfrak{h}^\cG(\mathrm{Spec}(k))_\bbQ(-r)
\end{equation}
\begin{equation}\label{eq:morph-comp-2}
\oplus^{\mathrm{dim}(X)}_{r=0} \oplus_{j=1}^n \mathfrak{h}^\cG(\mathrm{Spec}(k))_\bbQ(-r) \too \mathfrak{h}^\cG(X)_\bbQ\,.
\end{equation}
The composition $\eqref{eq:morph-comp-2}\circ \eqref{eq:morph-comp-1}$ agrees with the $0^{\mathrm{th}}$-component of $\mathrm{g} \circ \mathrm{f} =\id$, \ie with the identity of $\mathfrak{h}^\cG(X)_\bbQ$. Therefore, since $\mathfrak{h}^\cG(\mathrm{Spec}(k))_\bbQ(-r)=\bbL^{\otimes r}$, the $\cG$-equivariant Chow motive $\mathfrak{h}^\cG(X)_\bbQ$ is a direct summand of $\oplus_{r=0}^{\mathrm{dim}(X)}\oplus_{j=1}^n \bbL^{\otimes r}$. By definition of the $\cG$-equivariant Lefschetz motive $\bbL$, we have $\Hom_{\Chow^\cG(k)_\bbQ}(\bbL^{\otimes p}, \bbL^{\otimes q})=\delta_{pq} \cdot \bbQ$, where $\delta_{pq}$ stands for the Kronecker symbol. This implies that $\mathfrak{h}^\cG(X)_\bbQ$ is a subsum of $\oplus_{r=0}^{\mathrm{dim}(X)} \oplus_{j=1}^n \bbL^{\otimes r}$. Using the fact that $\pi(\bbL^{\otimes r})$, resp. $\pi(\mathfrak{h}^\cG(X)_\bbQ)$,  is isomorphic to $\pi(\mathfrak{h}^\cG(\mathrm{Spec}(k))_\bbQ)$, resp. $\oplus^n_{j=1} \pi(\mathfrak{h}^\cG(\mathrm{Spec}(k))_\bbQ)$, we conclude finally that there exists a choice of integers $r_1, \ldots, r_n \in \{0, \ldots, \mathrm{dim}(X)\}$ such that $\mathfrak{h}^\cG(X)_\bbQ \simeq \bbL^{\otimes r_1}\oplus \cdots \oplus \bbL^{\otimes r_n}$. This concludes the proof.
\subsection{Permutations}\label{sub:permutation}
Given a subgroup $\cH \subseteq \cG$, consider the small $\cG$-dg category $\cG \circlearrowright \amalg_{\overline{\rho} \in \cG/\cH}k$, where $\cG$ acts by permutation of the components.
\begin{proposition}\label{prop:permutation}
Let $\cG \circlearrowright \cA$ be a small $\cG$-dg category such that $\cD_c(\cA)$ admits a full exceptional collection $(\cE_1, \ldots, \cE_n)$. Assume that the induced $\cG$-action on $\cD_c(\cA)$ transitively permutes the objects $\cE_1, \ldots, \cE_n$ (up to isomorphism) and that $\Hom(\cE_i, \cE_j[m])=0$ for every $m \in \bbZ$ and $i \neq j$. Let $\cH \subseteq \cG$ be the stabilizer of $\cE_1$. If the cohomology group $H^2(\cH,k^\times)$ is trivial (\eg\ $k=\bbC$ and $\cH$ cyclic), then we have an isomorphism $\cG \circlearrowright \cA \simeq \cG \circlearrowright \amalg_{\overline{\rho} \in \cG/\cH}k$ in $\Hmo^\cG(k)$.
\end{proposition}
\begin{proof}
Similarly to the proof of Proposition \ref{prop:adjunctions}, we have the equivalence:
\begin{eqnarray*}
(\Pi_{\overline{\rho} \in \cG/\cH}\cD_c(\cA))^\cG \too \cD_c(\cA)^\cH && (\{\mathrm{B}_{\overline{\rho}}\}_{\overline{\rho} \in \cG/\cH}, \{\theta_\sigma\}_{\sigma \in \cG}) \mapsto (\mathrm{B}_{\overline{1}}, \{\theta_\sigma\}_{\sigma \in \cH})\,.
\end{eqnarray*}
Consequently, we obtain an induced identification 
\begin{eqnarray}\label{eq:induced-relative}
&& \quad \,\, \Hom(U^\cG(\cG \circlearrowright \amalg_{\overline{\rho} \in \cG/\cH}k), U^\cG(\cG \circlearrowright \cA))\simeq \Hom(U^\cH(\cH \circlearrowright_1 k), U^\cH(\cH \circlearrowright \cA))\,.
\end{eqnarray}
Since by assumption the cohomology group $H^2(\cH,k^\times)$ is trivial, the $\cH$-invariant object $\cE_1$ is $\cH$-equivariant, \ie it belongs to $\cD_c(\cA)^\cH$; see Remark \ref{rk:2-cocycle}. Via the identification \eqref{eq:induced-relative}, $\cE_1$ corresponds then to a morphism $
\cG \circlearrowright \amalg_{\overline{\rho} \in \cG/\cH}k \to \cG \circlearrowright \cA$ in $\Hmo^\cG(k)$. Using the fact that $\Hom_{\cD_c(\cA)}(\cE_i, \cE_j[m])=0$ for every $m \in \bbZ$ and $i \neq j$, we observe that this morphism is a $\cG$-equivariant Morita equivalence. Therefore, the proof follows now automatically from Lemma \ref{lem:inverts}.
\end{proof}
\begin{proposition}\label{prop:blocks}
Let $X$ be a quasi-compact quasi-separated $\cG$-scheme such that $\perf(X)$ admits a full exceptional collection
\begin{equation}\label{eq:blocks}
\left(\cE_1^1, \ldots, \cE_1^{s_1}, \ldots, \cE_i^1, \ldots, \cE_i^{s_i}, \ldots, \cE_n^1, \ldots, \cE_n^{s_n}\right)\,.
\end{equation}
For every fixed $i \in \{1, \ldots, n\}$, assume that the $\cG$-action on $\perf(X)$ transitively permutes the objects $\cE_i^1, \ldots, \cE_i^{s_i}$ (up to isomorphism) and that $\Hom(\cE_i^j, \cE_i^l[m])=0$ for every $m \in \bbZ$ and $j \neq l$. Let $\cH_i \subseteq \cG$ be the stabilizer of $\cE^1_i$. If $\cH_i \neq \cG$, assume that the cohomology group $H^2(\cH_i, k^\times)$ is trivial. If $\cH_i=\cG$, denote by $[\alpha_i] \in H^2(\cG,k^\times)$ the cohomology class of Remark \ref{rk:2-cocycle} associated to the exceptional object $\cE^1_i$. Under these assumptions, we have an isomorphism $U^\cG(\cG \circlearrowright \perf_\dg(X)) \simeq \oplus_{i=1}^n U^\cG(\cG \circlearrowright \perf_\dg(X)_i)$ in $\Hmo_0^\cG(k)$ where 
$$ U^\cG(\cG \circlearrowright \perf_\dg(X)_i) \simeq \begin{cases} U^\cG(\cG \circlearrowright \amalg_{\overline{\rho} \in \cG/\cH_i}k) & \mathrm{if}\,\, \cH_i\neq \cG \\ U^\cG(\cG \circlearrowright_{\alpha_i} k) & \mathrm{if}\,\,\cH_i=\cG\,.
\end{cases}
$$

\end{proposition}

%
\begin{remark}
Note that in the particular case where $s_1=\cdots=s_n=1$, Proposition \ref{prop:blocks} reduces to Proposition \ref{prop:exceptional-schemes}.
\end{remark}
\begin{proof}
Let us denote by $\perf(X)_i$ the smallest triangulated subcategory of $\perf(X)$ generated by the exceptional objects $\cE_i^1, \ldots, \cE_i^{s_i}$. In the same vein, let us write $\perf_\dg(X)_i$ for the full dg subcategory of $\perf_\dg(X)$ consisting of those objects which belong to $\perf(X)_i$. Under these notations, the full exceptional collection \eqref{eq:blocks} can be written as a semi-orthogonal decompositon $\perf(X) = \langle \perf(X)_1, \ldots, \perf(X)_n\rangle$. Making use of Proposition \ref{prop:semiorthogonal}, we hence obtain an isomorphism between $U^\cG(\cG \circlearrowright \perf_\dg(X))$ and $\oplus^n_{i=1} U^\cG(\cG \circlearrowright \perf_\dg(X)_i)$ in $\Hmo^\cG_0(k)$. The proof follows now from Proposition \ref{prop:permutation}, resp. Proposition \ref{prop:exceptional-schemes}, applied to each one of the $\cG$-dg categories such that $\cH_i\neq \cG$, resp. $\cH_i =\cG$.
\end{proof}
\begin{example}[Even dimensional quadrics]\label{ex:evenquadrics}
Let $Q_q$ be a smooth projective quadric of even dimension $d$; consult Example \ref{ex:quadrics}. As proved by Kapranov in \cite{Kapranov}, $\perf(Q_q)$ admits a full exceptional collection $(S_-, S_+, \cO, \cO(1), \ldots, \cO(d-1))$, where $\cS_+$ and $\cS_-$ denote the spinor bundles. Moreover, we have $\Hom(S_-,S_+[m])=0$ for every $m \in \bbZ$. Similarly to Example \ref{ex:quadrics}, the objects $\cO(i)$ are $\cG$-invariant for any $\cG$-action on $Q_q$. In what concerns the spinor bundles, they are $\cG$-invariant or sent to each other by the quotient $\cG/\cH \simeq C_2$; see \cite[\S3.2]{Elagin-semiorthogonal}. In the former case, we obtain a motivic decomposition similar to the one of Example \ref{ex:quadrics}. In the latter case, assuming that $H^2(\cH, k^\times)$ is trivial, Proposition \ref{prop:blocks} yields an isomorphism between $U^\cG(\cG \circlearrowright \perf_\dg(Q_q))$ and the direct sum
$$ U^\cG(\cG \circlearrowright \amalg_{\overline{\rho} \in C_2} k) \oplus U^\cG(\cG \circlearrowright_1 k) \oplus U^\cG(\cG \circlearrowright_\alpha k) \oplus \cdots \oplus U^\cG(\cG \circlearrowright_{\alpha^{(d-1)}}k)\,,$$
where $[\alpha]$ stands for the cohomology class of Remark \ref{rk:2-cocycle} associated to $\cO(1)$.

%
\end{example}
\begin{example}[del Pezzo surfaces]\label{ex:Pezzo}
Assume that $\mathrm{char}(k)=0$. Let $X$ be the del Pezzo surface obtained by blowing up $\bbP^2$ at two distinct points $x$ and $y$. As proved by Orlov in \cite[\S4]{Orlov}, $\perf(X)$ admits a full exceptional collection of length five $(\cO_{E_1}(-1), \cO_{E_2}(-1), \cO, \cO(1), \cO(2))$, where $E_1:=\pi^{-1}(x)$ and $E_2:=\pi^{-1}(y)$ denote the exceptional divisors of the blow-up $\pi\colon X \to \bbP^2$. Moreover, we have $\Hom(\cO_{E_1}(-1), \cO_{E_2}(-1)[m])=0$ for every $m \in \bbZ$. The objects $\cO(i)$ are $\cG$-invariant for every $\cG$-action on $X$. In what concerns $\cO_{E_1}(-1)$ and $\cO_{E_2}(-1)$, they are $\cG$-invariant or sent to each other by the quotient $\cG/\cH \simeq C_2$; see \cite[\S3.3]{Elagin-semiorthogonal}. In the former case, Proposition \ref{prop:exceptional-schemes} yields an isomorphism~between~$U^\cG(\cG \circlearrowright \perf_\dg(X))$~and
$$ U^\cG(\cG \circlearrowright_\gamma k)\oplus U^\cG(\cG \circlearrowright_\beta k) \oplus U^\cG(\cG \circlearrowright_1 k) \oplus  U^\cG(\cG \circlearrowright_\alpha k) \oplus U^\cG(\cG \circlearrowright_{\alpha^2} k)\,,$$
where $[\alpha]$, $[\beta]$, and $[\gamma]$, stand for the cohomology classes of Remark \ref{rk:2-cocycle} associated to the exceptional objects $\cO(1)$, $\cO_{E_2}(-1)$, and $\cO_{E_1}(-1)$, respectively. In the latter case, assuming that the cohomology group $H^2(\cH, k^\times)$ is trivial, Proposition \ref{prop:blocks} yields an isomorphism between $U^\cG(\cG \circlearrowright \perf_\dg(X))$ and the direct sum
$$ U^\cG(\cG \circlearrowright \amalg_{\overline{\rho} \in C_2} k)\oplus U^\cG(\cG \circlearrowright_1 k) \oplus U^\cG(\cG \circlearrowright_\alpha k) \oplus U^\cG(\cG \circlearrowright_{\alpha^2} k)\,.$$

\end{example}
\begin{remark}[Direct summands]
Let $X$ be a smooth projective $\cG$-scheme as in Proposition \ref{prop:blocks}. A proof similar to  Theorem \ref{thm:via} shows that $\mathfrak{h}(X)_\bbQ$ is a direct summand of the $\cG$-equivariant Chow motive $\oplus^{\mathrm{dim}(X)}_{r=0} \oplus^n_{i=0} \mathfrak{h}^\cG(\amalg_{\overline{\rho} \in \cG/\cH_i}\mathrm{Spec}(k))_\bbQ(-r)$, where $\cG$ acts by permutation of the components. 
\end{remark}
\section{Equivariant motivic measures}
In this section, by a {\em variety} we mean a reduced separated $k$-scheme of finite type. Let us write $\mathrm{Var}^\cG(k)$ for the category of {\em $\cG$-varieties}, \ie varieties which are equipped with a $\cG$-action such that every orbit is contained in an affine open set; this condition is automatically verified whenever $X$ is quasi-projective. The {\em Grothendieck ring of $\cG$-varieties $K_0\mathrm{Var}^\cG(k)$} is defined as the quotient of the free abelian group on the set of isomorphism classes of $\cG$-varieties $[X]$ by the relations $[X]=[Y]+[X\backslash Y]$, where $Y$ is a closed $\cG$-subvariety of $X$. The multiplication is induced by the product of $\cG$-varieties (with diagonal $\cG$-action). A {\em $\cG$-equivariant motivic measure} is a ring homomorphism $\mu^\cG\colon K_0\mathrm{Var}^\cG(k) \to R$. 
\begin{example}
\begin{itemize}
\item[(i)] When $k \subseteq \bbC$, the topological Euler characteristic $\chi$ (with compact support) gives rise to a $\cG$-equivariant motivic measure
\begin{eqnarray*}
\mu_{\chi}^\cG\colon K_0\mathrm{Var}^\cG(k) \too R_\bbQ(\cG) && [X] \mapsto \sum_i (-1)^i H_c^i(X^{\mathrm{an}},\bbQ)\,,
\end{eqnarray*}
where $H^i_c(X^{\mathrm{an}},\bbQ)$ is a finite dimensional $\bbQ$-linear $\cG$-representation;
\item[(ii)] When $\mathrm{char}(k)=0$, the characteristic polynomial $P_X(t):=\sum_i H_{dR}^i(X)t^i$, with $X$ a smooth projective $\cG$-variety, gives rise to a $\cG$-equivariant motivic measure $\mu_P^\cG\colon K_0 \mathrm{Var}^\cG(k) \to R(\cG)[t]$, where $H^i_{dR}(X)$ is considered as a finite dimensional $k$-linear $\cG$-representation.
\end{itemize}
\end{example}
Let us denote by $K_0(\NChow^\cG(k))$ the Grothendieck ring of the additive symmetric monoidal category of $\cG$-equivariant noncommutative Chow motives.
\begin{theorem}\label{thm:measure}
When $\mathrm{char}(k)=0$, the assignment $X \mapsto [U^\cG(\cG \circlearrowright \perf_\dg(X))]$, with $X$ a smooth projective $\cG$-variety, gives rise to a $\cG$-equivariant motivic measure
$$ \mu^\cG_{\mathrm{nc}}\colon K_0\mathrm{Var}^\cG(k) \too K_0(\NChow^\cG(k))\,.$$
\end{theorem}
\begin{proof}
Thanks to Bittner's presentation of the ring $K_0\mathrm{Var}^\cG(k)$ (see \cite[Lem.~7.1]{Bittner}), it suffices to verify the following two conditions:
\begin{itemize}
\item[(i)] Given smooth projective $\cG$-schemes $X$ and $Y$, we have:
$$ [U^\cG(\cG \circlearrowright \perf_\dg(X \times Y))]=[U^\cG(\cG \circlearrowright \perf_\dg(X)) \otimes U^\cG(\cG \circlearrowright \perf_\dg(Y))]\,.$$
\item[(ii)] Let $X$ be a smooth projective $\cG$-variety, $Y$ a closed smooth $\cG$-subvariety of codimension $c$, $\mathrm{Bl}_Y(X)$ the blow-up of $X$ along $Y$, and $E$ the exceptional divisor of this blow-up. Under these notations, the difference
$$ [U^\cG(\cG \circlearrowright \perf_\dg(\mathrm{Bl}_Y(X)))]- [U^\cG(\cG \circlearrowright \perf_\dg(E))]$$
is equal to the difference
$$[U^\cG(\cG \circlearrowright \perf_\dg(X))]- [U^\cG(\cG \circlearrowright \perf_\dg(Y))]\,.$$
\end{itemize}
As proved in \cite[Lem.~4.26]{Gysin}, we have the $\cG$-equivariant Morita equivalence
\begin{eqnarray*}
\perf_\dg(X) \otimes \perf_\dg(Y) \too \perf_\dg(X\times Y) && (\cF,\mathcal{G}) \mapsto \cF \boxtimes \mathcal{G}\,.
\end{eqnarray*}
Therefore, condition (i) follows from the combination of Lemma \ref{lem:inverts} with the fact that the functor $U^\cG$ is symmetric monoidal. In what concerns condition (ii), recall from Orlov \cite[Thm.~4.3]{Orlov} that $\perf_\dg(\mathrm{Bl}_Y(X))$ contains full $\cG$-dg subcategories $\perf_\dg(X), \perf_\dg(Y)_0, \ldots, \perf_\dg(Y)_{c-2}$ inducing a semi-orthogonal decomposition $\perf(\mathrm{Bl}_Y(X))=\langle \perf(X), \perf(Y)_0, \ldots, \perf(Y)_{c-2}\rangle$. Moreover, we have an isomorphism $\perf_\dg(Y)_i\simeq \perf_\dg(Y)$ in $\Hmo^\cG(k)$ for every $i$. Making use of Proposition \ref{prop:semiorthogonal}, we obtain the equality 
\begin{equation*}
[U^\cG(\cG \circlearrowright \perf_\dg(\mathrm{Bl}_Y(X)))]=[U^\cG(\cG \circlearrowright \perf_\dg(X))]+ (c-1)[U^\cG(\cG \circlearrowright \perf_\dg(Y))]\,.
\end{equation*}
Similarly, recall from \cite[Thm.~2.6]{Orlov} that $\perf_\dg(E)$ contains full $\cG$-dg subcategories $\perf_\dg(Y)_0, \ldots, \perf_\dg(Y)_{c-1}$ inducing a semi-orthogonal decomposition $\perf(E)=\langle \perf(Y)_0, \ldots, \perf(Y)_{c-1}\rangle$. Moreover, $\perf_\dg(Y)_i \simeq \perf_\dg(Y)$ in $\Hmo^\cG(k)$ for every $i$. Making use of Proposition \ref{prop:semiorthogonal}, we conclude that 
$$ [U^\cG(\cG \circlearrowright \perf_\dg(E))] = c [U^\cG(\cG \circlearrowright \perf_\dg(Y))]\,.$$
Condition (ii) follows now automatically from the preceding two equalities.
\end{proof}
\begin{proposition}\label{prop:factorization-measure}
The motivic measure $\mu_\chi^\cG\otimes_\bbQ \bbC$ factors through $\mu^\cG_{\mathrm{nc}}$.
\end{proposition}
\begin{proof}
Hochschild homology $HH\colon \dgcat(k) \to \cD(k)$ is an example of a symmetric monoidal additive invariant. Thanks to Proposition \ref{prop:factorization-enhaced}, it gives then rise to an additive symmetric monoidal functor $\overline{HH^\circlearrowright}\colon \Hmo_0^\cG(k) \to \cD(k)^\cG$ such that $\overline{HH^\circlearrowright}\circ U^\cG\simeq HH^\circlearrowright$. Consider the following composition
\begin{equation}\label{eq:comp-last}
\Hmo_0^\cG(k) \stackrel{\overline{HH^\circlearrowright}}{\too} \cD(k)^\cG \stackrel{-\otimes_k \bbC}{\too} \cD(\bbC)^\cG\,.
\end{equation}
It is well-known that an object of $\cD(k)$ is strongly dualizable if and only if it is compact. Since the category of $\cG$-equivariant noncommutative Chow motives is rigid (see Proposition \ref{prop:rigid}), the composition \eqref{eq:comp-last} yields a ring homomorphism 
\begin{equation}\label{eq:induced}
K_0(\NChow^\cG(k)) \too K_0(\cD_c(\bbC)^\cG) \simeq R_\bbC(\cG)\,.
\end{equation}
We claim that $\mu_\chi^\cG\otimes_\bbQ \bbC$ agrees with the composition of $\mu_{\mathrm{nc}}^\cG$ with \eqref{eq:induced}. Let $X$ be a smooth projective $\cG$-variety. Thanks to Bittner's presentation of $K_0\mathrm{Var}^\cG(k)$, it suffices to verify that the class of $HH^\circlearrowright(\cG \circlearrowright \perf_\dg(X))\otimes_k \bbC$ in the representation ring $R_\bbC(\cG)$ agrees with $\sum_i (-1)^i H^i_c(X^{\mathrm{an}}, \bbC)$. This follows from the identifications
\begin{eqnarray}
[HH^\circlearrowright (\cG \circlearrowright \perf_\dg(X))\otimes_k \bbC] & = & \sum_i (-1)^i HH_i(\perf_\dg(X))\otimes_k \bbC \nonumber \\
& = & \sum_i (-1)^i \oplus_{p-q=i} H^q(X, \Omega_X^p)\otimes_k \bbC \label{eq:HKR} \\ \nonumber
& = & \sum_{p,q} (-1)^{p-q} H^q(X,\Omega_X^p)\otimes_k \bbC \\ \nonumber
& = & \sum_{p,q} (-1)^{p+q} H^q(X,\Omega_X^p)\otimes_k \bbC\\ \nonumber
& = & \sum_i (-1)^i H^i_c(X^{\mathrm{an}}, \bbC)\,,\\ \nonumber
\end{eqnarray}
where \eqref{eq:HKR} is a consequence of the (functorial) Hochschild-Kostant-Rosenberg isomorphism $HH_i(\perf_\dg(X)) \simeq \oplus_{p-q=i} H^q(X, \Omega^p_X)$.
%
%
\end{proof}

\medbreak\noindent\textbf{Acknowledgments:} The author is very grateful to Dmitry Kaledin, Maxim Kontsevich and Michel Van den Bergh for stimulating discussions. He also would like to thank the Institut des Hautes \'Etudes Scientifiques (IH\'ES) for its hospitality and excellent working conditions, where this work was initiated.

\end{document}

\end{proof}